\documentclass[a4paper,12pt]{article}
\usepackage{amssymb}
\usepackage{amsthm}
\usepackage{tikz,url}
\usepackage{tikz,pgfplots}
\usetikzlibrary{decorations.markings}
\usepackage{amsmath,amssymb}
\usepackage{todonotes}
\usepackage{subcaption}

\usepackage{color}
\usepackage{graphicx}
\usepackage{placeins,caption}

\usetikzlibrary{calc}

\DeclareMathOperator{\strongprod}{\,\boxtimes\,}
\DeclareMathOperator{\mob}{Mob_{gp}}
\DeclareMathOperator{\mobmv}{Mob_{\mu}}

\DeclareMathOperator{\cmobmv}{Mob^*_{\mu}}
\DeclareMathOperator{\cmob}{Mob^*_{gp}}
\DeclareMathOperator{\gp}{gp}
\DeclareMathOperator{\diam}{diam}

\DeclareMathOperator{\cir}{c}
\DeclareMathOperator{\cp}{\,\square\,}
\usepackage[top=3cm,bottom=3cm,left=2.8cm,right=2.8cm]{geometry}
\newtheorem{theorem}{Theorem}[section]
\newtheorem{lemma}[theorem]{Lemma}
\newtheorem{corollary}[theorem]{Corollary}
\newtheorem{proposition}[theorem]{Proposition}

\newtheorem{conjecture}[theorem]{Conjecture}
\newtheorem{problem}[theorem]{Problem}

\theoremstyle{definition}

\newtheorem{definition}[theorem]{Definition}

\newcommand{\move}{\rightsquigarrow}

\newcommand{\address}[1]{#1}

\tikzset{middlearrow/.style={
		decoration={markings,
			mark= at position 0.75 with {\arrow[scale=2]{#1}} ,
		},
		postaction={decorate}
	}
}

\tikzset{midarrow/.style={
		decoration={markings,
			mark= at position 0.75 with {\arrow[scale=2]{#1}} ,
		},
		postaction={decorate}
	}
}

\begin{document}
	
	\title{Solution to some conjectures on mobile position problems}
	
	\author{ Ethan Shallcross $^{a}$ \\ \texttt{\footnotesize ershallcross@gmail.com} \\
		\and
		James Tuite $^{a,b}$ \\ \texttt{\footnotesize james.t.tuite@open.ac.uk} \\
		\and Aoise Evans $^{a}$ \\ \texttt{\footnotesize aoiseev@gmail.com}
		\and
		Aditi Krishnakumar $^{a}$ \\ \texttt{\footnotesize aditikrishnakumar@gmail.com} \\  
		\and Sumaiyah Boshar   \\ \texttt{\footnotesize }    
	}
	
	\maketitle	
	
	\address{
		\noindent
		$^a$  School of Mathematics and Statistics, Open University, Milton Keynes, UK\\
		$b$ Department of Informatics and Statistics, Klaip\.{e}da University, Lithuania\\
		
	}

	\begin{abstract}
		The general position problem for graphs asks for the largest number of vertices in a subset $S \subseteq V(G)$ of a graph $G$ such that for any $u,v \in S$ and any shortest $u,v$-path $P$ we have $S \cap V(P) = \{ u,v\} $, whereas the mutual visibility problem requires only that for any $u,v \in S$ there exists a shortest $u,v$-path with $S \cap V(P) = \{ u,v\} $. In the mobile versions of these problems, robots must move through the network in general position/mutual visibility such that every vertex is visited by a robot. This paper solves some open problems from the literature. We quantify the effect of adding the restriction that every robot can visit every vertex (the so-called \emph{completely mobile} variants), prove a bound on both mobile numbers in terms of the clique number, and find the mobile mutual visibility number of line graphs of complete graphs, strong grids and Cartesian grids.
	\end{abstract}
	
	
	\medskip\noindent
	{\bf AMS Subj.\ Class.\ (2020)}: 05C12, 05C69
	
	\section{Introduction}
	
	Imagine that a swarm of robots is stationed at the nodes of a network, and they communicate by sending signals to each other along shortest paths. To keep communication channels free, we require that for any pair of robots there is at least one shortest path between them that does not pass through another robot. Finding the largest number of robots that can satisfy this condition is equivalent to the \emph{mutual visibility problem} for graphs~\cite{DiStefano}. More strictly, if we require that \emph{all} shortest paths between any pair of robots are free from a third robot, we obtain the \emph{general position problem} for graphs (see~\cite{survey} for a survey of this problem). Interestingly this communication problem was one of the motivations for the problem introduced in the arXiv version of one of the seminal papers for the general position problem for graphs~\cite{manuel-2017arxiv}, whereas the first paper~\cite{DiStefano} on the mutual visibility problem for graphs was inspired by an already large literature on mutual visibility problems in robotic navigation.
	
	One feature missing from both the general position and mutual visibility problems is motion. The paper~\cite{klavzar-2023} considered a dynamic variation on the general position problem in which a swarm of robots must move through the graph, with one robot moving to a free adjacent vertex at each stage, such that each vertex is visited by a robot and the swarm of robots remain in general position at each stage. This is called the \emph{mobile general position problem} and the largest number of robots in such a swarm is the \emph{mobile general position number} of the graph. The paper~\cite{klavzar-2023} discussed the relationship between the mobile general position number and the general position number, determined the mobile general position number of Kneser graphs, cycles, complete multipartite graphs, block graphs and line graphs of complete graphs, and showed that the mobile general position number of unicyclic graphs is unbounded. 
	
	It follows from the realisation result for mobile general position vs. general position in~\cite{klavzar-2023} that there is no lower bound for the mobile general position number in terms of the general position number. A problem posed in~\cite{klavzar-2023} was whether there exists a lower bound for the mobile general position number as a function of the clique number.
	
	\begin{problem}[\cite{klavzar-2023}]\label{prob:mob gp vs clique num}
		Is there a general relationship between the mobile general position number and
		clique number?
	\end{problem}
	
	A further paper~\cite{Cartesian} explored the mobile general position number for Cartesian products, corona products and joins of graphs. In the section on join graphs, it was proved that $\min \{ \omega(G),\omega (H)\} +1 \leq \mob (G \vee H) \leq \omega (G)+\omega (H)-1$ when $G$ and $H$ both have order at least two and at least one is not complete, whilst for a join with a trivial graph we have $2 \leq \mob (G \vee K_1) \leq \omega (G)+1$ (where both the upper and lower bounds are tight). It was left as an open question which graphs can meet the lower bound, in particular whether it is possible to meet it if the clique number of $G$ is large. 
	
	\begin{problem}[\cite{Cartesian}]\label{prob:join}
		Are there graphs $G$ with arbitrarily large clique number such that $\mob (G \vee K_1) = 2$?   
	\end{problem}
	
	The paper~\cite{dettlaff2024} introduced the mobile version of the mutual visibility problem. As well as finding the mobile mutual visibility number of various graph classes, the authors showed that a graph has mobile mutual visibility number two if and only if it is a tree, investigated the problem in lexicographic, strong and Cartesian products and showed that the mobile mutual visibility problem is NP-hard.
	
	The mobile mutual visibility number of strong grids were classified in Problem~\ref{prob:strong grids} up to two possible values, and left finding the exact value as an open problem. The authors also found the mobile mutual visibility number of some prism graphs, including the Cartesian grid $P_n \cp P_2$, and posed the problem of finding this number for all Cartesian grids.  
	
	\begin{problem}[Problem 6,~\cite{dettlaff2024}]\label{prob:strong grids}
		For $r \geq 4, s \geq 3$, the mobile mutual visibility number of the strong grid $P_r \strongprod P_s$ satisfies \[ 2r+2s-6 \leq \mobmv(P_r \strongprod P_s) \leq 2r+2s-5.\] Which value is correct?
	\end{problem}
	
	\begin{problem}[Problem 6,~\cite{dettlaff2024}]\label{prob:Cartesian grids}
		What is the mobile mutual visibility number of a Cartesian grid?
	\end{problem}
	
	The paper~\cite{dettlaff2024} also investigated mobile mutual visibility sets in the line graphs of complete graphs. Mutual visibility sets in $L(K_n)$ correspond to the edges of $K_4$-free subgraphs of $K_n$. It follows that the number of robots of a mobile mutual visibility set of $L(K_n)$ is bounded above by the size of the 3-partite Tur\'{a}n graph of order $n$. Using this observation, the authors showed that strict inequality holds in this bound, and asked how large the difference between the mobile mutual visibility number and this upper bound can be.
	
	\begin{problem}[Problem 5,~\cite{dettlaff2024}]\label{prob:LK_n}
		What is the difference between the mutual visibility number and the mobile mutual visibility number of $L(K_n)$, the line graph of a clique?
	\end{problem}
	
	A problem posed in~\cite{klavzar-2023} concerned the effect of adding the restriction that every vertex can be visited by every robot; this is called the \emph{completely mobile general position problem}. We can define the \emph{completely mobile general position number} in an analogous way to the mobile general position number. In particular,~\cite{klavzar-2023} asked how large the difference between the completely mobile and mobile general position numbers can be.
	
	\begin{problem}[\cite{klavzar-2023}]\label{prob:completely mobile}
		What is the result on the mobile general position number if we make the stronger requirement that \emph{every} vertex is visited by \emph{every} robot?
	\end{problem}
	
	The goal of the present paper is to answer all of these problems. A further question asked in~\cite{dettlaff2024} is whether the equality $\mobmv (Q_d) = \mu (Q_d)$ holds for any hypercube graph. We can immediately answer this question in the negative: computer search shows that $\mobmv (Q_4) = 8 < 9 = \mu (Q_4)$. The inequality can also be seen from the fact that all mutual visibility sets of $Q_4$ of order nine are equivalent up to isomorphism.
	
	The plan of this paper is as follows. Section~\ref{sec:preliminaries} introduces the necessary graph theory terminology and background for the general position and mutual visibility problems, and proves a realisation result for the mobile general position and mobile mutual visibility numbers. We answer Problem~\ref{prob:mob gp vs clique num} in Section~\ref{sec:clique number} by giving a tight lower bound for the mobile numbers as a function of the clique number, which also settles Problem~\ref{prob:join} in the negative. In Section~\ref{sec:join} we consider the mobile mutual visibility number of the join of graphs. Section~\ref{sec:grid graph} considers the mobile mutual visibility problem for strong and Cartesian grids, thus answering Problems~\ref{prob:strong grids} and~\ref{prob:Cartesian grids}. In Section~\ref{sec:completely mobile} we answer Problem~\ref{prob:mob gp vs clique num} by characterising all possible combinations of mobile and completely mobile position numbers for the general position and mutual visibility problems. Using a stability result for Tur\'{a}n's Theorem we resolve Problem~\ref{prob:LK_n} in Section~\ref{sec:line graph complete graph}. We conclude with some open problems in Section~\ref{sec:conclusion}.
	
	\section{Preliminaries}\label{sec:preliminaries}
	
	A graph $G = (V,E)$ consists of a set $V(G)$ of vertices connected by a set $E(G)$ of edges. All graphs considered in this paper are finite, simple and undirected. We denote the clique number of a graph by $\omega (G)$. The distance $d(u,v)$ between two vertices $u,v$ is the length of a shortest $u,v$-path, and the diameter $\diam (G)$ of $G$ is $\max \{ d(u,v):u,v \in V(G)\} $. For a subset $S \subseteq V(G)$ and a vertex $u \in V(G)$ the distance $d(u,S)$ between $u$ and $S$ is $\min \{ d(u,v):v \in S\} $. We will denote the difference of two sets $S$ and $T$ by $S-T$. We will sometimes identify the vertices of a path $P_n$ with $[n] = \{ 1,\dots ,n\} $ in the natural way. The length of a longest cycle in $G$ is the \emph{circumference} of $G$, which we denote by $\cir (G)$.
	
	The Cartesian product $G \cp H$ of two graphs $G$ and $H$ is the graph with vertex set $V(G) \times V(H)$ in which two vertices $(u,v),(u',v')$ are adjacent in $G \cp H$ if and only if either $u = u'$ and $v \sim_H v'$, or $v = v'$ and $u \sim_G u'$. Similarly, the strong product $G \strongprod H$ is the graph with vertex set $V(G) \times V(H)$ in which any pair $(u,v)$ and $(u',v')$ is adjacent if and only if $u = u'$ and $ v \sim _H v'$, $v = v'$ and $u \sim _G u'$, or $u \sim _G u'$ and $v \sim _H v'$. The line graph $L(G)$ of a graph $G$ is the graph with vertex set $E(G)$ in which any two $e_1,e_2 \in V(L(G)) = E(G)$ are adjacent if and only if they are incident to a common vertex in $G$. For any terminology not defined here, we refer to~\cite{Graphsdigraphs}. We now formally define general position and mutual visibility sets.
	
	\begin{definition}[General position]
		A subset $S \subseteq V(G)$ of a graph $G$ is a \emph{general position set} of $G$ if no shortest path of $G$ passes through more than two vertices of $S$. The largest cardinality of a general position set of $G$ is the \emph{general position number} of $G$, denoted by $\gp (G)$.
	\end{definition}
	
	\begin{definition}[Mutual visibility]
		A subset $S \subseteq V(G)$ of a graph $G$ is a \emph{mutual visibility set} if for any $u,v \in S$ there exists a shortest $u,v$-path $P$ of $G$ such that $V(P) \cap S = \{ u,v\} $. The largest cardinality of a mutual visibility set of $G$ is the \emph{mutual visibility number} of $G$, written $\mu (G)$.
	\end{definition}
	
	When we wish to make a general statement on general position and mutual visibility, we adopt the convention that $\pi $ can stand for either $\gp $ or $\mu $, as the case may be, and refer to a `position set' of that type as a $\pi $-set. Suppose then that $S$ is a $\pi $-set of $G$. A \emph{legal move} is when a vertex $u \in S$ is replaced by an adjacent vertex $u' \in N(u)-S$ such that $S' = (S - \{ u\} ) \cup \{ u'\} $ is also a $\pi $-set. We indicate this move by $u \move u'$. 
	
	\begin{definition}
		A \emph{configuration} in $G$ is an assignment of robots to distinct vertices of $G$. We will denote this as an ordered tuple $(u_1,u_2,\dots ,u_t)$, where robot $R_i$ is assigned to vertex $u_i$ for $1 \leq i \leq t$. The configuration is a \emph{mobile $\pi $-configuration} if there is a sequence of legal moves starting from $(u_1,u_2,\dots ,u_t)$ such that every vertex of $G$ is visited by a robot. The set of vertices corresponding to the configuration is a \emph{mobile $\pi $-set} and the number of vertices in a largest mobile $\pi $-set is the \emph{mobile general position number} $\mob (G)$ or the \emph{mobile mutual visibility number} $\mobmv (G)$, as the case may be.
	\end{definition}

	When clear from the context, we will sometimes identify a robot with the position that it occupies. We use the following characterisation of a mobile $\pi $-set, which allows us to show mobility without describing the full sequence of moves needed to move the swarm around every vertex of the graph. If we start from a $\pi $-configuration and can visit any vertex of $G$ with a robot by a sequence of legal moves, then we can reverse the process to return the robots to their starting configuration and then repeat the process to visit another vertex until all vertices have been visited by a robot.
	
	\begin{lemma}\label{lem:mobility}
		A configuration $(u_1,\dots ,u_t)$ in a graph $G$ is a mobile $\pi $-configuration if and only if for any $u \in V(G)$ there is a sequence of legal moves starting from $(u_1,\dots ,u_t)$ such that $u$ is visited by a robot, with the swarm of robots remaining a $\pi $-set at each stage.
	\end{lemma}
	
	We begin by comparing the mobile general position and mutual visibility numbers. As any general position set is also a mutual visibility set, it is trivial that $\gp (G) \leq \mu (G)$ and $\mob (G) \leq \mobmv (G)$ for any graph $G$. 
	
	\begin{proposition}\label{prop:realisation}
		There exists a graph $G$ with $\mob (G) = a$ and $\mobmv (G) = b$ if and only if $a = b = 1$ or $2 \leq a \leq b$.
	\end{proposition}
	\begin{proof}
		Trivially $\mob (G) = 1$ if and only if $\mobmv (G) = 1$, so assume that $a \geq 2$. If $a = b$, then $K_a$ suffices, so additionally we can take $a < b$. Consider the graph $G(a,b) = K_2 \vee ((b-a)K_1 \cup K_{a-1})$ of order $b+1$. We will denote the vertices of $(b-a)K_1 \cup K_{a-1}$ by $\{ u_1,\dots ,u_{b-a}\} \cup \{ v_1,\dots ,v_{a-1}\} $, where $u_1,\dots ,u_{b-a}$ are the $b-a$ isolated vertices and $v_1, \dots , v_{a-1}$ are the vertices of the clique $K_{a-1}$. Denote the additional two vertices by $x_1,x_2$. First observe that $V(G(a,b))-\{ x_2\} $ is a mobile mutual visibility set, as it is a mutual visibility set and the robot at $x_1$ can make the legal move $x_1 \move x_2$. As the graph is not complete, we must have $\mobmv (G(a,b)) \leq b$, so in fact $\mobmv (G(a,b)) = b$.
		
		The set $\{ x_1,v_1, \dots , v_{a-1}\} $ is a general position set, and the remaining vertices can be visited by the moves $x_1 \move x_2$ and $x_1 \move u_i$ for $1 \leq i \leq b-a$, so $\mob (G(a,b)) \geq a$. The upper bound $\mob (G(a,b)) \leq a$ follows immediately from~\cite[Theorem 4.4]{Cartesian}. This graph therefore has $\mob (G(a,b)) = a$ and $\mobmv (G(a,b)) = b$.
	\end{proof}
	The construction in Proposition~\ref{prop:realisation} is obviously smallest possible, although interestingly in general there appear to be many graphs with these parameters and minimum order.

	\section{Bounds on mobile position by clique number}\label{sec:clique number}
	
	In this section we first solve Problem~\ref{prob:mob gp vs clique num} by demonstrating a relationship between the clique number and the mobile general position number; in particular, we show that for graphs with fixed diameter the mobile general position number is $\Omega (\omega (G))$. We then use this result to answer Problem~\ref{prob:join}. Finally we give a lower bound for the mobile mutual visibility number in terms of the clique number.
	
	\begin{theorem}\label{thm:mobgp vs clique number} The completely mobile general position number of any non-complete graph $G$ with clique number $\omega(G) \geq 3$ is bounded below by 
		\[ \mob (G) \geq \mob ^*(G) \geq 1+\left \lceil \frac{\omega (G)}{\diam (G)} \right \rceil -\theta (s),\] where $\omega (G) = q\diam (G)+s$, where $0 \leq s \leq \diam (G)-1$, and $\theta (s) = 1$ if $s = 1$ and zero otherwise. The bound is tight.
	\end{theorem}
	\begin{proof}
		Let $W$ be a largest clique of $G$. Set $\omega = \omega (G)$, $D = \diam (G)$ and $\theta = \theta (s)$. Start with $k = 1+\left \lceil \frac{\omega }{D} \right \rceil -\theta $ robots in $W$. The robots can move freely within $W$, so every vertex within $W$ can be visited. Let $v_r$ be any vertex at distance $r \geq 1$ from $W$ and $P = v_0,v_1,\dots ,v_r$ be a shortest path from $W$ to $v_r$, where $v_0 \in W$. Label each vertex of $W$ with the sequence $(d(w,v_1),d(w,v_2),\dots ,d(w,v_r))$. In such a sequence each entry is either equal to or one more than the preceding entry, the $i$-th term of each sequence is either $i$ or $i+1$ for $1 \leq i \leq r$, and if $r = D$, then the $r$-th entry is $r$. It is easily seen that there are $r+1$ possible sequences for $1 \leq r \leq D-1$ (for example, for $r = 4 < \diam (G)$ the sequences are $(1,2,3,4)$, $(2,2,3,4)$, $(2,3,3,4)$, $(2,3,4,4)$ and $(2,3,4,5)$) and $D$ sequences when $r = D$. Therefore there are always at most $D$ sequences.
		
		Hence by the Pigeonhole Principle there is a subset $W' \subseteq W$ of order $\geq \left \lceil \frac{\omega }{D} \right \rceil $ such that all vertices of $W'$ have the same sequence. If a subset $W'$ corresponding to a sequence other than $(1,\dots ,r)$ has order $\geq \left \lceil \frac{\omega }{D} \right \rceil $, station $\left \lceil \frac{\omega }{D}\right \rceil $ robots in $W'$ and one robot $R$ at $v_0$. Otherwise, we can suppose that each such subset has order $< \left \lceil \frac{\omega }{D}\right \rceil$. If $2 \leq s < D$ and each subset apart from the $(1,\dots ,r)$-subset has order $\leq q$, then the $(1,\dots ,r)$-subset has order $\geq q+s \geq \left \lceil \frac{\omega }{D} \right \rceil +1$. Similarly, if $s = 0$, then we can assume that each subset apart from the $(1,\dots ,r)$-subset has order $\leq q-1$, so that the $(1,\dots ,r)$-subset has order at least \[ qD - (D-1)(q-1) = \frac{\omega }{D}+D-1 \geq \left \lceil \frac{\omega }{D}\right \rceil +1.\] If $s = 1$, we can only conclude that the $(1,\dots ,r)$-subset has order $\geq \left \lceil \frac{\omega }{D}\right \rceil $. Thus in any case we can station all $k$ robots inside the $(1,\dots ,r)$-subset, with one robot $R$ at $v_0$. Now send $R$ along the path $P$. At each stage $R$ will be equidistant from the robots in $W'$ and so after each of the $r$ moves the robots will occupy a general position set. 
		
		It follows that $k$ robots can visit every vertex of the graph. Finally observe that since the robots can be permuted at the initial stage inside $W$, this configuration is in fact completely mobile. The bound is tight, as evidenced by the graph formed from a clique $K_{n-1}$ and adding a new vertex adjacent to $\left \lceil \frac{n-1}{2} \right \rceil $ vertices of $K_{n-1}$.
		
	\end{proof}
	
	To show that Theorem~\ref{thm:mobgp vs clique number} is tight for diameter three, for all $\omega \geq 3$ we construct a graph $G(\omega )$ with $\diam (G(\omega )) = 3$, clique number $\omega $, and $\mob{(G(\omega ))} = \left \lceil \frac{\omega }{3} \right \rceil + 1 - \theta(s)$, where $\omega = 3q+s$, $q \geq 1$, $0 \leq s \leq 2$. To form $G(\omega )$, we begin with a disjoint union $K_{3q+s} \cup K_3$. Label the vertices in $K_{3q+s}$ as $a_1, a_2, \dots, a_{3q+s}$ and the vertices in $K_3$ as $1,2,3$. To complete the construction, add edges between $1$ and $a_i$ for $1 \leq i \leq q+1$ and between $2$ and $a_i$ for $q+1 \leq i \leq 2q+1$, omitting edge $2a_3$ when $q=1$ and $s=0$. For example, $G(7)$ is shown on the right of Figure~\ref{fig:T4.1Counter}. It is not difficult to verify that these graphs have the required parameters.
	
	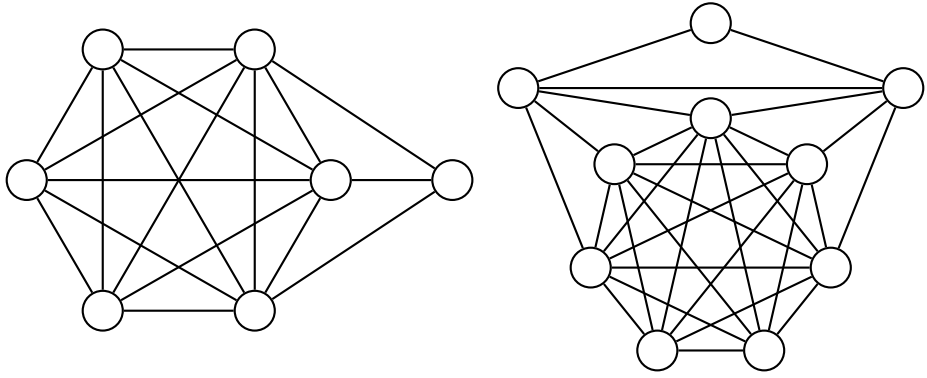
\begin{figure}
		\centering
		\begin{tikzpicture}[x=0.2mm,y=-0.2mm,inner sep=0.5mm,scale=1,thick,vertex/.style={circle,draw,minimum size=15}]
			
			\node at (0,-40) [vertex] (v1) {};
			\node at (-50,46.6) [vertex] (v2) {};
			\node at (-150,46.6) [vertex] (v3) {};
			\node at (-200,-40) [vertex] (v4) {};
			\node at (-150,-126.6) [vertex] (v5) {};
			\node at (-50,-126.6) [vertex] (v6) {};
			\node at (80,-40) [vertex] (v7) {};
			
			\path
			(v7) edge (v1)
			(v7) edge (v2)
			(v7) edge (v6)
			
			(v1) edge (v2)
			(v2) edge (v3)
			(v3) edge (v4)
			(v4) edge (v5)
			(v5) edge (v6)
			(v6) edge (v1)
			
			(v1) edge (v3)
			(v2) edge (v4)
			(v3) edge (v5)
			(v4) edge (v6)
			(v5) edge (v1)
			(v6) edge (v2)
			
			(v1) edge (v4)
			(v2) edge (v5)
			(v3) edge (v6)
			;
			
			\begin{scope}[yshift = 0cm, xshift = 5cm]
				
				\node at ($(0,0)!9!90:(9,0)$) [vertex] (a0) {};
				
				\foreach \j [count=\k, evaluate=\k as \i using {int(\k-1)}] in {1,...,6}{
					\node at ($(0,0)!9!\j*360/7+90:(9,0)$) [vertex] (a\j) {};
					\foreach \l in {0,...,\i}{
						\path (a\l) edge (a\j);
					}
				};
				
				\node at ($(0,0)!9!1*360/7+90:(18,0)$) [vertex] (b1) {};
				\node at ($(0,0)!9!6*360/7+90:(18,0)$) [vertex] (b2) {};
				
				\node at ($(0,0)!9!90:(16,0)$) [vertex] (c1) {};
				
				\path (b2) edge (a0);
				\path (b2) edge (a6);
				\path (b2) edge (a5);
				
				\path (b1) edge (a0);
				\path (b1) edge (a1);
				\path (b1) edge (a2);
				
				\path (b1) edge (c1);
				\path (b2) edge (c1);
				
				\path (b1) edge (b2);
				
			\end{scope}
		\end{tikzpicture}
		
		\caption{Families of graphs with diameters two and three that demonstrate tightness of Theorem~\ref{thm:mobgp vs clique number}.}
		\label{fig:T4.1Counter}
		
	\end{figure}
	
	We note also that as the random graph $G_{n,\frac{1}{2}}$ has diameter two and clique number $\geq 2\log (n)$ with high probability, Theorem~\ref{thm:mobgp vs clique number} implies the following lower bound.
	\begin{corollary}
		With high probability, $\mob (G_{n,\frac{1}{2}}) = \Theta (\log (n))$.
	\end{corollary}
	
	Whilst Theorem~\ref{thm:mobgp vs clique number} is tight for diameters two and three, it seems that it can be strengthened for fixed clique numbers. For example, for clique number five our formula only guarantees $\mob ^*(G) \geq 2$ for large diameter, whereas this can be improved to three.
	
	\begin{theorem}\label{the:clique num 5}
		For all graphs $G$ with $\omega(G) \geq 5$, $\mob(G) \geq \cmob(G) \geq 3$.
	\end{theorem}
	\begin{proof}
		Let $G$ be any graph with $\omega(G) \geq 5$ and let $W \subseteq V(G)$ be any largest clique. Begin by stationing three robots $R_1, R_2, R_3$ in $W$. Clearly, the robots can move freely through $W$ whilst remaining in general position, and can permute their positions inside $W$, so our proof will imply that the three robots are complete mobile. 
		
		If $V(G) = W$, then we are done. Otherwise, let $v \in V(G) - W$ and let $P$ be a shortest path $v_0,v_1, \dots, v_r = v$ from $W$ to $v$, where $v_0 \in W$. As observed in the preceding proof, for all $w \in W$ and $i \in [r]$, $d(w,v_i) = i$ or $i+1$. For each $w \in W$, define $j(w)$ to be the smallest $i$ in the range $1 \leq i \leq r$ such that $d(w, v_i) = i$ if such an integer exists and $\infty$ otherwise. If $|N(v_1) \cap W| \geq 3$, we can station $R_1, R_2, R_3$ in $N(v_1) \cap W$ and move $R_1$ along $P$ to reach $v$, whilst the robots remain in general position at each step. If there are distinct vertices $w_1,w_2 \in W-N(v_1)$ with $j(w_1) = j(w_2)$, then we can station $R_1, R_2, R_3$ at $v_0,w_1,w_2$, respectively and move $R_1$ along $P$ to reach $v$, and again the robots remain in general position.
		
		In any other case, we may choose $w_1,w_2,w_3 \in N(v_1)-W$ such that $2 \leq j(w_1) < j(w_2) < j(w_3)$. Set $\ell = j(w_2)$.  Position $R_1, R_2, R_3$ at $v_0,w_2,w_3$, respectively. For $1 \leq i \leq \ell-1$, $d(w_2,v_i) = d(w_3, v_i) = i+1$, so we may move $R_1$ along $P$ to $v_{\ell -1}$ whilst maintaining general position. 
		
		Denote a fixed shortest $w_2,v_{\ell }$-path by $Q$ and let $w_2w_2'$ be the first edge of $Q$. We must have $w_2' \not \in W$ and $d(w_2',v_{\ell }) = \ell - 1$. Since $d(w_3, v_\ell) = \ell + 1$, it follows that $d(w_3,w_2') = 2$. Also, $\ell - 1 \leq d(w_2', v_{\ell -1}) \leq d(w_2', v_\ell) + d(v_\ell, v_{\ell -1}) = \ell $ and $d(w_3, v_{\ell -1}) = \ell$, so $R_2$ can make the move $w_2 \move w_2'$. 
		
		If $w_1$ and $v_0$ are both neighbours of $w_2'$, then we could have positioned $R_1,R_2,R_3$ at $w_2,w_1,v_0$ respectively, and set $R_1$ along the path $Q$ and then the section of $P$ from $v_{\ell }$ to $v$, whilst the robots remain in general position. Hence, we can assume that at least one of the vertices $\{ v_0,w_1\} $ is not adjacent to $w_2'$; call this vertex $w^*$. For $\ell-1 \leq i \leq r$, $d(w^*,v_i) = i, d(w^*,w_2') = 2$ and for $\ell \leq i \leq r$, $d(w_2', v_i) = i-1$. Therefore, with $R_1,R_2,R_3$ at $v_{\ell -1},w_2',w_3$ respectively, we see that $R_3$ can make the move $w_3 \move w^*$ and then $R_1$ can proceed along $P$ from $v_{\ell -1}$ to $v$ whilst maintaining general position.
	\end{proof}

	\begin{figure}
		
		\centering
		
		\begin{tikzpicture}[x=0.2mm,y=-0.2mm,inner sep=0.5mm,scale=1,thick,vertex/.style={circle,draw,minimum size=15}]
			
			\node at ($(0,0)!9!90:(9,0)$) [vertex] (a0) {};
			
			\foreach \j [count=\k, evaluate=\k as \i using {int(\k-1)}] in {1,...,4}{
				\node at ($(0,0)!9!\j*360/5+90:(9,0)$) [vertex] (a\j) {};
				\foreach \l in {0,...,\i}{
					\ifnum \l>0
					\path (a\l) edge (a\j);
					\else
					\ifnum \j>1
					\path (a\l) edge (a\j);
					\fi
					\fi
				}
			};
			
			\node at ($(0,0)!9!90:(18,0)$) [vertex] (b1) {};
			\node at ($(0,0)!9!360/5+90:(18,0)$) [vertex] (b2) {};
			
			\node at ($(0,0)!9!360/10+90:(24,0)$) [vertex] (c1) {};
			
			\path (c1) edge (b1);
			\path (c1) edge (b2);
			
			\path (b1) edge (b2);
			\path (b1) edge (a0);
			\path (b1) edge (a1);
			\path (b1) edge (a4);
			
			\path (b2) edge (a0);
			\path (b2) edge (a1);
			\path (b2) edge (a2);
			
		\end{tikzpicture}
		\caption{A graph $G$ with a subgraph $K_5^-$ for which $\mob{G =} 2$.}
		\label{fig:K5-mobgp2}
	\end{figure}
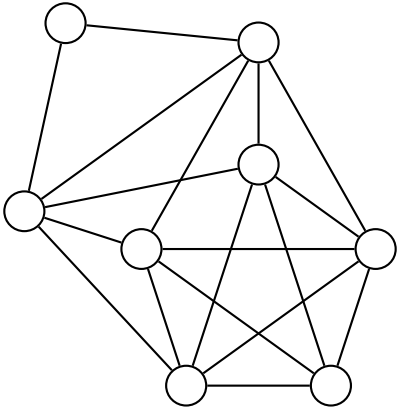
	
	Let $T(n,r)$ be the $r$-partite Tur\'{a}n graph with order $n$ and $t(n,r)$ be the size of $T(n,r)$. Theorem~\ref{the:clique num 5} immediately shows that the largest size of a graph with mobile general position number two is at most $t(n,4)$. However, based on computations up to order nine~\cite{Erskine}, we make the following conjecture.
	
	\begin{conjecture}\label{conj:maxsize mobile num 2}
		For $n \geq 3$, $T(n,3)$ is the unique graph with order $n$, mobile general position number two and maximum size. 
	\end{conjecture}
	
	Conjecture~\ref{conj:maxsize mobile num 2} would follow immediately if Theorem~\ref{the:clique num 5} if any graph containing $K_4$ has mobile general position number at least three. However, this is false; in fact the graph in Figure~\ref{fig:K5-mobgp2} contains a $K_5^-$.
	
	We can now apply Theorem~\ref{thm:mobgp vs clique number} to answer Problem~\ref{prob:join} in the negative.
	
	\begin{corollary}
		Let $G$ be any graph with $\omega(G) \geq 4$. Then $\mob(G \vee K_1) \geq 3$.
	\end{corollary}
	\begin{proof}
		This follows immediately from Theorem~\ref{the:clique num 5} by noting that for any graph $G$ with $\omega(G) \geq 4$, $\omega(G \vee K_1) \geq 5$.
	\end{proof}
	
	Finally we derive a bound for the mobile mutual visibility number in terms of the clique number.

	\begin{theorem}\label{thm:cliquevsmobmvnum}
		For any graph $G$, the mobile mutual visibility number is at least $ \mobmv (G) \geq \omega (G)$ and the completely mobile mutual visibility number is at least $\mobmv ^*(G) \geq \omega (G)-1$.
	\end{theorem}
	\begin{proof}
		Let $W$ be a largest clique of $G$. We start with the robots in $W$, which is obviously a mutual visibility set. If $G$ is complete, we are done, otherwise let $u$ be any vertex at distance $r \geq 1$ from $W$ and $P$ be a shortest path $v_0,v_1,\dots ,v_r=u$, $v_0 \in W$ from $W$ to $u$. We send the robot $R$ at $v_0$ along the path $P$ until it reaches $u$. When $R$ reaches $v_k$, $k \leq r$, any vertex $w$ of $W-\{ v_0\} $ is either at distance $k$ from $v_k$, in which case there is a shortest $v_k,w$-path not passing through another vertex of $W-\{v_0\} $, or is at distance $k+1$ from $v_k$, in which case the edge $wv_0$ followed by the path $P$ is a shortest path not passing through a third robot. Therefore the configuration of robots is mutually visible at each stage. 
		
		If there are just $\omega (G)-1$ robots in $W$, then we can permute their positions using the spare vertex of $W$, then follow the same procedure as described above, so that such a configuration is completely mobile.
	\end{proof}
	
	Trivially the graph formed by adding a leaf to some or all of the vertices of a complete graph shows that both of these bounds are tight.

	\section{Mobile mutual visibility of joins}\label{sec:join}
	
	Mobile mutual visibility sets of the join $G \vee H$ of two graphs $G$ and $H$ were studied in~\cite{dettlaff2024}, in which it was shown that if $G$ and $H$ are both non-complete graphs with order at least two, then \[ n (G)+ n(H)-3 \leq \mobmv (G \vee H) \leq n(G)+n(H)-1.\] In this section we will sharpen the lower bound, and also find the exact answer for all joins. First we will need a classification of the graphs with order $n$ and mobile mutual visibility number $n-1$. Graphs with mutual visibility number equal to $n-1$ were classified by Di Stefano in~\cite{DiStefano}.
	
	\begin{definition}
		For a graph $G$ with order $n$, a subset $U \subseteq V(G)$ of cardinality $k$ is a \emph{$k$-hub} if any pair of non-adjacent vertices $w_1,w_2 \in V(G)-U$ has a common neighbour in $U$. We will call a $1$-hub simply a \emph{hub}.
	\end{definition}
	
	\begin{lemma}[\cite{DiStefano}]\label{lem:hub}
		A subset $S$ of $n-1$ vertices of a graph with order $n$ is a mutual visibility set if and only if the only vertex of $G$ outside of $S$ is a hub.
	\end{lemma}
	
	\begin{theorem}\label{thm:mobmv=n-1}
		A non-complete graph $G$ has $\mobmv (G) = n-1$ if and only if it contains a pair of adjacent hubs.
	\end{theorem}
	\begin{proof}
		If there is a pair $u_1,u_2$ of adjacent hubs in $G$, then we can station robots on the vertices of $V(G)-\{ u_2\} $ and then move a robot back and forth between $u_1$ and $u_2$.
		
		For the converse, suppose that we have $n-1$ mutually visible robots that can traverse $G$. Let $u$ be the unoccupied vertex in the initial configuration of the robots; by Lemma~\ref{lem:hub}, $u$ is a hub. In the first move a robot at some vertex $v$ must make the move $v \move u$, at which point $v$ will be left unoccupied. Thus $v$ is also a hub and $u \sim v$. 
	\end{proof}
	
	Considering the placement of hubs, we can classify the mobile mutual visibility numbers of all joins.
	
	\begin{corollary}
		Let $G$ and $H$ be connected graphs of order at least two that are not both complete. Then $\mobmv (G \vee H) = n(G)+n(H)-1$ if and only if
		\begin{itemize}
			\item either $\mobmv (G) \geq n(G)-1$ or $\mobmv (H) \geq n(H)-1$, or
			\item $\mu (G) = n(G)-1$ and $\mu (H) = n(H)-1$.
		\end{itemize}
		Otherwise, $\mobmv(G \vee H) = n(G)+n(H)-2$. 
	\end{corollary}
	\begin{proof}
		Note that any configuration of robots on $G \vee H$ will be mutually visible if there is at least one unoccupied vertex in both $G$ and $H$. By Theorem~\ref{thm:mobmv=n-1}, the join $G \vee H$ will satisfy $\mobmv (G \vee H) = n(G)+n(H)-1$ if and only if $G \vee H$ contains a pair of adjacent hubs. Notice that a vertex of $G \vee H$ is a hub if and only if it is either a hub of $G$ or of $H$. Having two adjacent hubs in $G$ is equivalent to $\mobmv(G) \geq n(G)-1$, and similarly for $H$. Otherwise, having one hub in $G$ and one hub in $H$ is equivalent to $\mu (G) = n(G) -1$ and $\mu (H) = n(H)-1$.
		
		We now show that $\mobmv(G \vee H) \geq n(G)+n(H)-2$ for any graphs $G,H$ with order at least two, completing the proof. Begin by stationing $n(G) - 1$ robots in $G$ and $n(H) - 1$ robots in $H$. Move one of the robots in $G$ to the free vertex in $H$ and then do the same for $H$. All vertices have now been visited and at each step there was a vacant vertex in both $G$ and $H$.
	\end{proof}
	
	Finally we deal with joins in which one of the factors has order one. By $\mu _r(G)$ we denote the largest cardinality of a mutual visibility set $M$ of $G$ such that no two vertices of $M$ are at distance greater than $r$.
	
	\begin{theorem}
		For any graph $G$, the mobile mutual visibility number of the join $G \vee K_1$ is given by $\mobmv(G \vee K_1) = \mu _2(G) + 1$. 
	\end{theorem}
	\begin{proof}
		Let $v$ be the single vertex of $K_1$. Whenever $v$ is not occupied, any configuration will be mutually visible. Consider the stage that a robot $R$ is at $v$. Then there cannot be a pair of robots $R_1,R_2$ in $V(G)$ at distance three or more apart in $G$, since the unique shortest path in $G \vee K_1$ between them passes through $R$ at $v$. Furthermore, if $R_1$ and $R_2$ are stationed at non-adjacent vertices in $V(G)$, then they must have an unoccupied common neighbour in $G$. It follows that the robots in $V(G)$ at this point occupy a mutual visibility set of $G$ and are at distance at most two apart. The upper bound $\mobmv (G \vee K_1) \leq \mu _2(G)+1$ follows. From this configuration the robot at $v$ can visit any unoccupied vertex of $G$ in one step, so we have equality.
	\end{proof}

	\section{Mobile mutual visibility in grid graphs}\label{sec:grid graph}
	
	Our main goal in this section is to find the exact value of the mobile mutual visibility number of strong and Cartesian grids, thereby answering Problem 6 of~\cite{dettlaff2024}. In fact, whilst it was shown in~\cite{Cartesian} that $\mob (P_r \cp P_s) = 3$ for $r,s \geq 3$, the mobile general position number of strong grids has not appeared in the literature, so we begin by finding this value.

	\begin{proposition}
		For all $r \geq s \geq 3$, $\mob(P_r \strongprod P_s) = 4$.
	\end{proposition}
	\begin{proof}
		We assume that $V(P_r) = [r]$, $V(P_s) = [s]$, and $V(P_r \strongprod P_s) = [r] \times [s]$. First, it is known from~\cite[Section 4]{Klavzar-2019} that $\gp(P_r \strongprod P_s) = 4$. Therefore, it only remains to show that $\mob(P_r \strongprod P_s) \geq 4$.
		
		First, we consider the case where $r=s$. Start by stationing robots at $(1,1)$, $(1,s)$, $(s,s)$, and $(s,1)$. The robot at $(1,s)$ may traverse all vertices in $\{ (i,j) : 1 \leq i < j \leq s \}$ whilst maintaining general position by moving along a shortest path from $(1,s)$ to any such vertex. Similarly, the robot at $(s,1)$ may traverse $\{ (i,j) : 1 \leq j < i \leq s \}$. We must now visit the vertices $\{ (i,i) : 1 < i < s \}$. First, move all robots back to their original positions. Perform the moves $(1,s) \move (1, s-1) \move (1, s-2) \move \dots \move (1,2)$ and  $(s,1) \move (s-1, 1) (s-2, 1) \move \dots \move (2,1)$. Finally, move $(s,s) \move (s-1, s-1) \move \dots \move (2,2)$.
		
		Now suppose that $r \geq s+1$. Begin by stationing robots at $(1,1)$, $(1,s)$, $(s,s)$, and $(s,1)$. Perform the moves described previously to visit each vertex in $\{ (i,j) : 1 \leq i, j \leq s \}$. Then undo these moves and return the robots to their initial positions. For each $0 \leq t \leq r-s-1$, when the robots are stationed at $(t,1)$, $(t,s)$, $(s+t,s)$, and $(s+t, 1)$, the sequence \[(s+t,1) \move (s+t+1, 1), (s+t, s) \move (s+t+1, s), (t, 1) \move (t+1, 1), (t, s) \move (t+1, s)\] is legal. Thus, by repeatedly applying the moves described above (shifted appropriately) before performing this sequence, we see that each vertex in $P_r \strongprod P_s$ may be traversed whilst maintaining general position. We conclude that $\mob(P_r \strongprod P_s) = 4$.
	\end{proof}
	
	\begin{theorem}
		For any $3 \leq m \leq n$, the strong grid $P_n \boxtimes P_m$ has mobile mutual visibility number $2(n+m)-6$, except that $\mobmv (P_3 \boxtimes P_3) = 5$.
	\end{theorem}
	\begin{proof}
		It was shown in~\cite{dettlaff2024} that $2(n+m)-6 \leq \mobmv (P_n \boxtimes P_m) \leq 2(n+m)-5$, except for $m = n = 3$, in which case the largest mobile mutual visibility set contains five robots. We show that $P_n \boxtimes P_m$ has no mobile mutual visibility set of order $2(n+m)-5$. Suppose for a contradiction that $2(n+m)-5$ mutually visible robots can traverse the vertices of the strong grid. Consider the stage at which a robot is situated at the vertex $(2,2)$. By the \emph{star} at a vertex $(i,j)$ we will mean the union of the ascending and descending diagonals through $(i,j)$. Consider the star at $(2,2)$, together with the descending diagonals from the vertices $(1,i)$ and $(j,m)$, where $4 \leq i \leq m-1$ and $1 \leq j \leq n-2$. The star can contain at most three robots, since the robot at $(2,2)$ lies on the intersection of descending and ascending diagonals, and each descending diagonal contains at most two robots. Hence altogether these sets of vertices can hold at most $3+2(n+m-6) = 2(n+m)-9$ robots.  
		
		Suppose that $n = m$. Then there are four vertices not contained in the union of the star and the descending diagonals, namely $(1,2)$, $(2,1)$, $(n-1,m)$ and $(n,m-1)$. Hence when there is a robot at $(2,2)$, a robot must be stationed at each of these four vertices, and also the star at $(2,2)$ contains exactly three robots and each of the descending diagonals in our collection hold exactly two robots. Observe that if a robot was stationed at the ascending diagonal through $(2,2)$, at a vertex other than $(1,1)$ or $(2,2)$, then it would lie on both the star and one of the descending diagonals, so that we would have to reduce our count and there would be strictly less than the required $2(n+m)-5$ robots. Hence for the star to hold three robots the set $\{ (1,1),(1,3),(3,1) \}$ must hold two robots. We cannot have robots at both vertices $(1,3)$ and $(3,1)$, as this would leave us with three robots on a diagonal, so there must be a robot at $(1,1)$. However, the robot at $(1,1)$ is now surrounded by the robots at $(1,2)$, $(2,2)$ and $(2,1)$, so no other robots are visible to the robot at $(1,1)$, a contradiction. 
		
		The same argument holds if $n = m+1$. Our counting argument implies that the star at $(2,2)$ holds exactly three robots, each of the descending diagonal holds two robots, and a further four robots are stationed at $(1,2)$, $(2,1)$, $(n,m-1)$ and $(n,m)$. As before, for the star to hold three robots there must be a robot at $(1,1)$, which cannot see any robots beyond those at $(1,2)$, $(2,2)$ and $(2,1)$.
		
		Assume now that $n \geq m+2$. Now there are five vertices that do not lie in the star at $(2,2)$ or one of our collection of descending diagonals, namely $(1,2)$, $(2,1)$, $(n-1,m)$, $(n,m-1)$ and $(n,m)$, and at least four robots must be stationed on these five vertices. Suppose that there are robots at both $(1,2)$ and $(2,1)$. Then the star cannot contain a robot at $(1,1)$, as it would be surrounded as before. Furthermore, there is no robot at $(1,3)$, for otherwise such a robot would not be able to see the robot at $(2,1)$, and similarly there is no robot at $(3,1)$. Thus either the star contains just one robot, or else the star contains two robots, one of which is on the ascending diagonal apart from the vertices $(1,1)$ and $(2,2)$, which would also lie on a descending diagonal. Hence we can reduce our robot count to $5+1+2(n+m-6) = 2(n+m)-6$, which is too small.
		
		Therefore we conclude that there is a robot on exactly one of $(1,2)$ and $(2,1)$, robots at each of $(n-1,m)$, $(n,m-1)$ and $(n,m)$, three robots on the star and two robots on each of the descending diagonals. Again, if there is a robot on the intersection of the star and one of the descending diagonals, then there would be too few robots, so there are two robots on the set $\{ (1,3),(3,1),(1,1)\} $. To avoid having three robots on the descending diagonal through $(2,2)$ there must be a robot at $(1,1)$. 
		
		Suppose that the robot on $\{ (1,2),(2,1)\} $ is at $(2,1)$. Then the vertices $(3,1)$, $(3,2)$ and $(4,1)$ are not visible to the robot at $(1,1)$, so, considering the three robots in the star and the two robots on the descending diagonal from $(1,4)$, we see that there must be robots at each of $(1,3)$, $(1,4)$ and $(2,3)$; however, this means that the robot at $(1,4)$ is not visible to the robot at $(1,1)$. A symmetrical argument applies if there is a robot at $(1,2)$. We have thus shown that there is no mutual visibility set of $P_n \boxtimes P_m$ of order at least $2(n+m)-5$ that contains the vertex $(2,2)$.
		
	\end{proof}
	
	We now determine the mobile mutual visibility number of Cartesian grids. It was shown in~\cite[Proposition 6.2]{dettlaff2024} that $\mobmv (P_2 \cp P_n) = 3$ when $2 \leq n \leq 3$ and $4$ for $n \geq 4$. We first find the answer for grids $P_n \cp P_m$ where $n,m \geq 5$. In these two proofs, we will refer to a move $(x,y) \move (x-1,y)$ as moving the robot left, $(x,y) \move (x+1,y)$ as moving right, $(x,y) \move (x,y+1)$ as moving up and $(x,y) \move (x,y-1)$ as moving down.

	\begin{theorem}
		For $n \geq m \geq 5$, \[\mobmv(P_n \cp P_m) =
		\begin{cases}
			2m, & \text{if}\ n > m, \\
			2m-1, & \text{if}\ n = m.
		\end{cases}
		\]
		\label{largecartgrids}
	\end{theorem}
	\begin{proof}
		Let $n \geq m \geq 5$. We start by showing that $\mobmv(P_{m+1} \cp P_m) = 2m$ when $m$ is odd. Let $f(t)=1+|\left \lceil \frac{m}{2} \right \rceil - t|$. Begin by stationing robots on $(f(i),i)$ and $(m+2-f(i), i)$ for $1 \leq i \leq m$. Label the robot with first coordinate $1$ as $R_0$, the robot with first coordinate $2$ and largest second coordinate as $R_1$ and continue in a clockwise fashion to label the remaining robots $R_2, R_3, \dots, R_{2m-1}$. For $1 \leq k \leq \left \lfloor \frac{m}{2} \right \rfloor - 1$, move $R_1, R_2, \dots, R_{m+k-1}$ left and then (working modulo $2m$) $R_{m+k}, R_{m+k+1}, \dots, R_{2m-1+k}$ to the right. At the end of these moves, there are robots stationed at $(2, m)$ and $(m,1)$ which may be moved to $(1,m)$ and $(m+1,1)$, respectively.
		
		By symmetry, these moves show how to visit all vertices except $(\frac{m+1}{2}, \frac{m+1}{2})$ and $(\frac{m+3}{2}, \frac{m+1}{2})$. We show how to visit $(\frac{m+3}{2}, \frac{m+1}{2})$ starting from the original configuration. Move $R_{m-1}$ to the right and then $R_m$ to the left. In increasing order, move $R_1, R_2, \dots, R_{m-2}$ and $R_{m}$ to the left. If $m=5$, we are done. Otherwise, move $R_{m+1}, \dots, R_{2m-1}, R_{0}$ to the right. Now for $1 \leq k \leq \left \lfloor \frac{m}{2} \right \rfloor - 3$, move $R_{k+1}, \dots, R_{m+k-1}$ left, $R_{m+k+1}, R_{m+k+2}, \dots, R_{2m-1}$ and $R_0, \dots R_{k-1}$ to the right. Finally, set $k = \left \lfloor \frac{m}{2} \right \rfloor - 2$ and perform the sequence up until $R_m$ is moved to $(\frac{m+3}{2}, \frac{m+1}{2})$.
		
		\begin{figure}
			\captionsetup{justification=centering}
			\centering
			\begin{subfigure}{.4\textwidth}
				\centering
				
				\begin{tikzpicture}[x=0.2mm,y=-0.2mm,inner sep=0.5mm,scale=1,thick,vertex/.style={circle,draw,minimum size=10},
					occupied/.style={fill=black},
					occupied20/.style = {occupied},
					occupied11/.style = {occupied},
					occupied02/.style = {occupied},
					occupied03/.style = {occupied},
					occupied14/.style = {occupied},
					occupied25/.style = {occupied},
					occupied34/.style = {occupied},
					occupied43/.style = {occupied},
					occupied42/.style = {occupied},
					occupied31/.style = {occupied},]
					
					\foreach \r in {0,...,4}{
						\foreach \c in {0,...,5}{
							\node at (\c*30,5*30-\r*30) [vertex,occupied\r\c/.try] (\r\c) {};
						}
					};
					
					\foreach \r [count=\nr] in {0,...,3} {
						\foreach \c [count=\nc] in {0,...,4}{
							\path (\r\c) edge (\r\nc);
							\path (\r\c) edge (\nr\c);
						}
					};
					
					\foreach \c [count=\nc] in {0,...,4}{
						\path (4\c) edge (4\nc);
					}
					
					\foreach \r [count=\nr] in {0,...,3}{
						\path (\r5) edge (\nr5);
					}
					
				\end{tikzpicture}
				\caption[justification=centering]{The initial configuration for $n = 6, m = 5$.}
				\label{fig:a5r1initial}
			\end{subfigure}
			\begin{subfigure}{.4\textwidth}
				\centering
				
				\begin{tikzpicture}[x=0.2mm,y=-0.2mm,inner sep=0.5mm,scale=1,thick,vertex/.style={circle,draw,minimum size=10},
					occupied/.style={fill=black},
					occupied41/.style = {occupied},
					occupied21/.style = {occupied},
					occupied12/.style = {occupied},
					occupied42/.style = {occupied},
					occupied30/.style = {occupied},
					occupied03/.style = {occupied},
					occupied33/.style = {occupied},
					occupied24/.style = {occupied},
					occupied04/.style = {occupied},
					occupied15/.style = {occupied},]
					
					\foreach \r in {0,...,4}{
						\foreach \c in {0,...,5}{
							\node at (\c*30,5*30-\r*30) [vertex,occupied\r\c/.try] (\r\c) {};
						}
					};
					
					\foreach \r [count=\nr] in {0,...,3} {
						\foreach \c [count=\nc] in {0,...,4}{
							\path (\r\c) edge (\r\nc);
							\path (\r\c) edge (\nr\c);
						}
					};
					
					\foreach \c [count=\nc] in {0,...,4}{
						\path (4\c) edge (4\nc);
					}
					
					\foreach \r [count=\nr] in {0,...,3}{
						\path (\r5) edge (\nr5);
					}
					
				\end{tikzpicture}
				\caption[justification=centering]{The result of the first round of moves for $n = 6, m = 5$.}
				\label{fig:a5r1round1}
			\end{subfigure}
		\end{figure}
		
		By symmetry, all vertices can be visited while maintaining mutual visibility with $2m$ robots. Since two robots per row is the maximum possible for mutual visibility, $\mobmv(P_{m+1} \cp P_m) = 2m$.
		
		To see that $\cmobmv(P_{m+1} \cp P_m) = 2m - 1$, first notice that $\cmobmv(P_{m+1} \cp P_m) < 2m$, since when there are two robots in each row, no robot can move to another row. Begin by stationing the robots as previously described, with $R_m$ removed. Notice that we may move $R_{a+1}$ right and then up to occupy the vertex previously occupied by $R_{m}$. Therefore, by performing the moves described twice and ignoring the moves for $R_m$ and then $R_{m+1}$, we see that this is a mobile mutual visibility set. Furthermore, we may permute the robots by moving $R_i$ right then up for $m \leq i \leq \lfloor \frac{m}{2} \rfloor$. Then, move $R_{m + \lceil m/2 \rceil}$ right and $R_i$ down then right for $m + \lfloor m/2 \rfloor + 1 \leq i \leq 2m$. Next, move $R_i$ left then down for $1 \leq i \leq \lfloor \frac{m}{2} \rfloor$, $R_{ \lceil m/2 \rceil}$ left and $R_j$ up then left for $\lceil \frac{m}{2} \rceil + 1 \leq j \leq m+1$ (ignoring $R_m$).
		
		The proof for even $m$ is similar. However, we now put $f(t)= 1+|\frac{m}{2} - t|$ and begin by stationing robots at $(f(i),i)$ for $1 \leq i \leq m$, $(m+1-f(i), i)$ for $1 \leq i \leq \frac{m}{2}$, and $(m+3-f(i), i)$ for $\frac{m}{2} + 1 \leq i \leq m$. Notice that after the first sequence, the configuration of robots is the reflection in the vertical of the original. After the first group of repetitions, we are instead left with $(\frac{m}{2}+1, \frac{m}{2})$ and $(\frac{m}{2}+1, \frac{m}{2} + 1)$ unvisited. We may follow the same method as before to traverse these. Again two robots per row is the maximum possible. Thus, $\mobmv(P_{m+1} \cp P_m) = 2m$. A similar pattern of moves may be followed to permute the robots with $R_m$ removed, giving $\cmobmv(P_{m+1} \cp P_m) = 2m - 1$.
		
		\begin{figure}[h]
			
			\centering
			
			\captionsetup{justification=centering}
			
			\begin{tikzpicture}[x=0.2mm,y=-0.2mm,inner sep=0.2mm,scale=1,thick,vertex/.style={circle,draw,minimum size=10},
				occupied/.style={fill=black},
				occupied20/.style = {occupied},
				occupied11/.style = {occupied},
				occupied31/.style = {occupied},
				occupied02/.style = {occupied},
				occupied42/.style = {occupied},
				occupied03/.style = {occupied},
				occupied53/.style = {occupied},
				occupied54/.style = {occupied},
				occupied14/.style = {occupied},
				occupied25/.style = {occupied},
				occupied45/.style = {occupied},
				occupied36/.style = {occupied},]
				
				\foreach \r in {0,...,5}{
					\foreach \c in {0,...,6}{
						\node at (\c*30,6*30-\r*30) [vertex,occupied\r\c/.try] (\r\c) {};
					}
				};
				
				\foreach \r [count=\nr] in {0,...,4} {
					\foreach \c [count=\nc] in {0,...,5}{
						\path (\r\c) edge (\r\nc);
						\path (\r\c) edge (\nr\c);
					}
				};
				
				\foreach \c [count=\nc] in {0,...,5}{
					\path (5\c) edge (5\nc);
				}
				
				\foreach \r [count=\nr] in {0,...,4}{
					\path (\r6) edge (\nr6);
				}
				
			\end{tikzpicture}
			\caption[justification=centering]{The initial configuration for $n = 7, m = 6$.}
			\label{fig:a6r1initial}
			
		\end{figure}
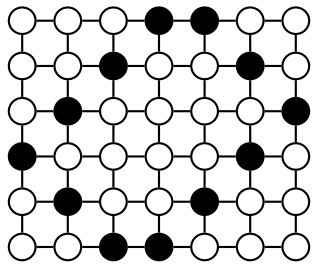
		
		Now if we have $n \geq m+2$ we can traverse all vertices with first coordinate at most $m+1$ by using the methods above. We can then move each robot (in decreasing order of the first coordinate) one space to the right before repeating the moves (appropriately shifted) to traverse the next column. By repeating this process, we again see that $\mobmv(P_{n} \cp P_m) = 2m$.
		
		This only leaves the case of $n = m$. In this case $\mobmv(P_n \cp P_n) < 2n$ since in any mutual visibility set of $P_n \cp P_n$ of size $2n$, there would be exactly two robots in each row and column. Therefore, no move would be possible from any initial configuration. 
		
		First consider when $n$ is even. Put $f(t) = 1 + |\left \lfloor \frac{n}{2} \right \rfloor + 1 - t|$ and station robots at $(f(i), i)$ for $1 \leq i \leq n$, $(n+2-f(i), i)$ for $2 \leq i \leq \frac{n}{2}$, and $(n+1-f(i),i)$ for $\frac{n}{2} < i \leq n$. Label the robots as before and for $1 \leq k \leq \lceil \frac{n}{2} \rceil - 2$ move $R_k, R_{k+1}, \dots, R_{n-2+k}$ to the left and then (working modulo $2n-1$) $R_{n+k}, R_{2n-2+k}$ to the right. There will then be a robot at $(2,n)$, which may be moved to $(1,n)$. By symmetry, all vertices may be visited in this way.
		
		We also have $\cmobmv(P_n \cp P_n) = 2n-1$. We show how to permute the robots in their initial configuration. Move $R_i$ down for $i=\frac{3n}{2}, \frac{3n}{2} + 1, \dots, 0$. Next, move $R_i$ left then down for $i=1,\dots,\frac{n}{2} - 1$ and $R_{n/2}$ left. After this, move $R_i$ up then right for $i=\frac{n}{2} + 1, \dots, n - 1$. Move $R_n$ up and then $R_i$ right then up for $i=n+1, \dots, \frac{3n}{2} - 1$. Finish by moving $R_i$ right for $i=\frac{3n}{2}, \frac{3n}{2} + 1, \dots, 0$.
		
		Finally, when $n$ is odd, station robots at $(f(i),i)$ for $1 \leq i \leq n$, $(n+1-f(i),i)$ for $2 \leq i \leq \left \lceil \frac{n}{2} \right \rceil$, and $(n+2-f(i),i)$ for $\left \lceil \frac{n}{2} \right \rceil < i \leq n$. Proceed as before to visit all vertices except $(\lceil \frac{n}{2} \rceil, \lceil \frac{n}{2} \rceil)$.
		
		\begin{figure}
			\captionsetup{justification=centering}
			\centering
			\begin{subfigure}{.4\textwidth}
				\centering
				
				\begin{tikzpicture}[x=0.2mm,y=-0.2mm,inner sep=0.5mm,scale=1,thick,vertex/.style={circle,draw,minimum size=10},
					occupied/.style={fill=black},
					occupied20/.style = {occupied},
					occupied11/.style = {occupied},
					occupied31/.style = {occupied},
					occupied02/.style = {occupied},
					occupied42/.style = {occupied},
					occupied13/.style = {occupied},
					occupied43/.style = {occupied},
					occupied24/.style = {occupied},
					occupied34/.style = {occupied},]
					
					\foreach \r in {0,...,4}{
						\foreach \c in {0,...,4}{
							\node at (\c*30,5*30-\r*30) [vertex,occupied\r\c/.try] (\r\c) {};
						}
					};
					
					\foreach \r [count=\nr] in {0,...,3} {
						\foreach \c [count=\nc] in {0,...,3}{
							\path (\r\c) edge (\r\nc);
							\path (\r\c) edge (\nr\c);
						}
					};
					
					\foreach \c [count=\nc] in {0,...,3}{
						\path (4\c) edge (4\nc);
					}
					
					\foreach \r [count=\nr] in {0,...,3}{
						\path (\r4) edge (\nr4);
					}
					
				\end{tikzpicture}
				\caption[justification=centering]{The initial configuration for $P_5 \cp P_5$.}
				\label{fig:a5r0initial}
			\end{subfigure}
			\begin{subfigure}{.4\textwidth}
				\centering
				
				\begin{tikzpicture}[x=0.2mm,y=-0.2mm,inner sep=0.5mm,scale=1,thick,vertex/.style={circle,draw,minimum size=10},
					occupied/.style={fill=black},
					occupied30/.style = {occupied},
					occupied21/.style = {occupied},
					occupied41/.style = {occupied},
					occupied12/.style = {occupied},
					occupied52/.style = {occupied},
					occupied53/.style = {occupied},
					occupied03/.style = {occupied},
					occupied14/.style = {occupied},
					occupied44/.style = {occupied},
					occupied41/.style = {occupied},
					occupied25/.style = {occupied},
					occupied35/.style = {occupied},]
					
					\foreach \r in {0,...,5}{
						\foreach \c in {0,...,5}{
							\node at (\c*30,6*30-\r*30) [vertex,occupied\r\c/.try] (\r\c) {};
						}
					};
					
					\foreach \r [count=\nr] in {0,...,4} {
						\foreach \c [count=\nc] in {0,...,4}{
							\path (\r\c) edge (\r\nc);
							\path (\r\c) edge (\nr\c);
						}
					};
					
					\foreach \c [count=\nc] in {0,...,4}{
						\path (5\c) edge (5\nc);
					}
					
					\foreach \r [count=\nr] in {0,...,4}{
						\path (\r5) edge (\nr5);
					}
					
				\end{tikzpicture}
				\caption[justification=centering]{The initial configuration for $P_6 \cp P_6$.}
				\label{fig:a6r0initial}
			\end{subfigure}
		\end{figure}
		
		If $n = 5$, to traverse $(3,3)$ from the initial configuration, move $R_1$ left, $R_0$ right, $R_8$ left and then down, $R_7$ left then up, $R_6$ down, and then $R_0$ right to $(3,3)$. If $n > 5$, move $R_1$ left, $R_0$ right, $R_{2a-2}$ left, $R_{2n-\left \lfloor n/2 \right \rfloor+i}$ down for $0 \leq i \leq \left \lfloor\frac{n}{2}\right \rfloor - 2$. Finally, for $2 \leq i \leq \left \lfloor \frac{n}{2} \right \rfloor$, move $R_{2n-1-i}$ left, $R_0$ right, and then $R_i$ left. These steps result in $R_0$ at $(\left \lceil \frac{n}{2} \right \rceil, \left \lceil \frac{n}{2} \right \rceil)$. We conclude that $\mobmv(P_n \cp P_n) = 2n-1$.
		
		We can use similar moves to show that $\cmob(P_n \cp P_n) = 2n-1$ for odd $n$ as well.
	\end{proof}

	We now complete the classification of the mobile mutual visibility numbers of Cartesian grids by solving the remaining case of $3 \times n$ and $4 \times n$ grids.
	
	\begin{theorem}
		For $a \geq 3$,		 
		\[\mobmv(P_n \cp P_3) =
		\begin{cases}
			4, & \text{if}\ n = 3, \\
			5, & \text{if}\ n = 4, \\
			6, & \text{if}\ n \geq 6.
		\end{cases} \] and \[ \mobmv(P_n \cp P_4) =
		\begin{cases}
			5, & \text{if}\ n = 3, \\
			7, & \text{if}\ n = 4 \,\text{or}\ 5 \\
			8, & \text{if}\ n \geq 6.
		\end{cases}
		\]
		\label{smallcartgrids}
	\end{theorem}
	\begin{proof}
		To prove the upper bounds, we find the cardinalities of largest mutual visibility sets containing a central vertex of the grid.
		
		Firstly, suppose we had a mutual visibility set $S$ of $P_3 \cp P_3$ with $|S| = 5$, $(2,2) \in S$. We can suppose that $2$ vertices in $S$ have second coordinate $1$. In order to maintain mutual visibility, $(2,2)$ can be the only vertex in $S$ with second coordinate $2$. So there are $2$ vertices in $S$ with second coordinate $3$. It is easily seen that these vertices cannot be mutually visible and so $\mobmv(P_3 \cp P_3) \leq 4$.
		
		To see that $\mobmv(P_3 \cp P_3) = 4$, station robots at $(1,2), (2,1), (2,3), (3,2)$ and move $(2,1) \move (3,1)$, $(2,3) \move (1,1)$, and then $(1,2) \move (2,2)$. Since the positions of the robots may be permuted by moving them around the edge of the grid, in fact $\cmobmv(P_3 \cp P_3) = 4$.
		
		Next, suppose we had a mutual visibility set $S$ of $P_4 \cp P_3$ with $|S| = 6$, $(2,2) \in S$. Exactly two vertices in $S$ have each possible second coordinate. By the above reasoning, there can be at most four vertices in $S$ with first coordinate $<4$, so exactly two vertices in $S$ have first coordinate $4$. Similarly if $(3,2) \in S$, there are exactly two vertices in $S$ with first coordinate $1$ and the vertices cannot be mutually visible. If $(1,2) \in S$, there can only be at most one in $S$ with second coordinate $1$, a contradiction. Finally, if $(4,2) \in S$ we can assume without loss of generality that $(4,1) \in S$ and $(4,3) \notin S$. But then $(4,1)$ and the vertex in $S$ with second coordinate $3$ and smallest first coordinate will not be mutually visible. Thus, there can be no such $S$ and $\mobmv(P_4 \cp P_3) \leq 5$.
		
		To establish equality, station robots at $(1,2), (1,3), (2,1), (4,2), (4,3)$ and move $(1,2) \move (2,2)$, $(2,1) \move (1,1)$, and then $(1,3) \move (2,3)$. Since the positions of the robots may be permuted by moving them around the edge of the grid, in fact $\cmobmv(P_4 \cp P_3) = 4$.
		
		It also follows that any mutual visibility set of $P_4 \cp P_4$ containing $(1,1)$ can have at most five vertices with first coordinate $<4$ and so $\mobmv(P_4 \cp P_4) \leq 7$.
		
		To see that $\mobmv(P_4 \cp P_4) = 7$, station robots at the vertices \[ (1,2), (1,3), (2,4), (3,1), (3,4), (4,2), (4,3)\] and move $(1,2) \move (2,2)$, $(2,4) \move (1,4)$. From the initial configuration we may move $(3,1) \move (2,1)$ or $(1,2), (1,1), (2,1)$ and deduce the result by symmetry. We can use moves similar to those in Theorem~\ref{largecartgrids} to permute the robots, so that $\cmobmv(P_4 \cp P_4) = 7$.
		
		Now suppose we had a mutual visibility set $S$ of $P_5 \cp P_4$ with $|S| = 8$, $(3,2) \in S$. Exactly two vertices in $S$ have each possible second coordinate. If $(2,2) \in S$, we also have $(1,1), (1,3), (5,1), (5,3) \in S$. But then $S$ cannot contain two vertices with second coordinate $4$, a contradiction. If $(1,2) \in S$ then as two vertices in $S$ have second coordinate $1$, $(1,1) \in S$. But similar reasoning gives $(1,3) \in S$, a contradiction. Hence, $\mobmv(P_4 \cp P_5) \leq 7$. Equality follows from the result for $P_4 \cp P_4$ by applying the `shifting' argument used in the proof of Theorem~\ref{largecartgrids}. We also see that $\cmobmv(P_5 \cp P_4) = 7$.
		
		For $n \geq 2$, since at most two vertices in any mutual visibility set can have the same second coordinate, it is clear that $\mobmv(P_n \cp P_3) \leq 6$ and $\mobmv(P_n \cp P_4) \leq 8$.
		
		To see that $\mobmv(P_5 \cp P_3) = 6$, position six robots as in the proof of Theorem~\ref{largecartgrids}. Then perform the moves $(3,1) \move (4,1) \move (5,1)$, $(3,3) \move (4,3) \move (5,3)$, $(2,1) \move (1,1)$, $(1,2) \move (2,2)$, $(2,3) \move (1,3)$ and then $(2,2) \move (3,2)$. Using similar moves similar to those in Theorem~\ref{largecartgrids} to permute the robots, we see that $\cmobmv(P_5 \cp P_3) = 5$.
		
		Finally, we show that $\mobmv(P_6 \cp P_4) = 8$. Position eight robots as in the proof of Theorem~\ref{largecartgrids}. We may follow the sequences in that proof to traverse all vertices except $(3,2), (3,3)$, and $[1] \times [4]$. However, by moving the robots to the right in decreasing order of the first coordinates, we see that the remaining vertices can be traversed. Moves similar to those in Theorem~\ref{largecartgrids} show that $\mobmv(P_6 \cp P_4) = 7$ and $\mobmv(P_5 \cp P_3) = 5$.
	\end{proof}
	
	\section{Completely mobile position sets}\label{sec:completely mobile}
	
	Imagine a swarm of robots delivering groceries to customers. If the robots form a mobile general position set, then they can communicate freely and every house can be visited by a robot. However, if some houses can only be visited by certain robots, then it would be difficult to send the right groceries with the right robot, it may require a long time for some houses to be visited and some robots may exceed their carrying capacity. 
	
	Therefore, a more realistic requirement is that \emph{every} robot can visit \emph{every} vertex. We call the mobile general position problem with this added restriction the \emph{completely mobile general position problem}. The largest number of robots in a completely mobile general position set of $G$ is the \emph{completely mobile general position number} of $G$ and we will denote it by $\mob ^*(G)$. Naturally we must assume the graph to be connected. The completely mobile mutual visibility number $\mobmv ^*(G)$ is defined analogously. We will use the following criterion to show that a mobile $\pi $-configuration is completely mobile.

	\begin{lemma}\label{lem:swapping}
		A mobile $\pi $-configuration $(u_1,\dots ,u_t)$ is a completely mobile $\pi $-configuration if for any $1 \leq i,j \leq t$ there is a sequence of legal moves that returns the robots to the set $\{ u_1,\dots ,u_t\} $, but with the positions of the robots permuted so that robot $R_i$ is at vertex $u_j$.
	\end{lemma}
	
	In Subsection~\ref{subsec:completely mobile gp} we discuss completely mobile general position sets and we treat completely mobile mutual visibility sets in Subsection~\ref{subsec:completely mobile mv}.
	
	\subsection{Completely mobile general position sets}\label{subsec:completely mobile gp}
	
	We begin by characterising the graphs with extremal values of the completely mobile general position number. 
	
	\begin{proposition}\label{prop:comp mob small values}
		A connected graph has $\mob ^*(G) = 1$ if and only if it is a path. All graphs with order $n \geq 2$ satisfy $\mob ^*(G) \leq n-1$, with equality only for complete graphs. A graph $G$ has $\mob (G) = n$ if and only if $G$ is a clique, and $\mob (G) = n-1$ if and only if $G$ is isomorphic to a $K_{n-1}$ with a leaf attached.
	\end{proposition}
	\begin{proof}
		If there are two robots $R_1$ and $R_2$ on a path $P_n$, robot $R_2$ can never visit the leaf closest to $R_1$, so $\mob ^* (P_n) = 1$. Suppose that a connected graph $G$ contains a vertex $u$ with $\deg (u) \geq 3$. Then to show that $\mob ^*(G) \geq 2$, it is sufficient by Lemma~\ref{lem:swapping} to show that any two robots $R_1,R_2$ in $G$ can swap positions. Suppose that the distance from $R_1$ to $u$ is no greater than the distance from $R_2$ to $u$. Move $R_1$ to $u$ along a shortest path, and then $R_2$ to $N(u)$, say to a neighbour $u_1$ of $u$. If $u_2,u_3$ are additional neighbours of $u$, then moving $R_1$ to $u_2$, $R_2$ to $u$ and then $u_3$, then $R_1$ to $u$ and then $u_1$, and then $R_2$ back to $u$, and finally retracing the steps that brought them to $\{ u,u_1\} $, has swapped their positions.
		
		For the remainder of the result we need only look at complete graphs and the graphs with $\gp (G) = n(G)-1$; it shown in~\cite{ullas-2016} that in the latter case $G$ is either i) a join of $K_1$ with a disjoint union of cliques, or ii) a complete graph $K_n$ with $k$ edges incident to a common vertex deleted for some $1 \leq k \leq n-2$. Complete graphs have a mobile general position set of order $n$, but only $n-1$ robots are completely mobile, since $n$ robots cannot make any move. Checking the remaining graphs is routine. 
	\end{proof}
	
	Trivially for any graph we have $\mob ^*(G) \leq \mob (G)$. For the graphs studied in~\cite{klavzar-2023} we have $\mob ^*(G) = \mob (G)$ or $\mob (G) -1$. In particular, it is easily seen that the configurations given in~\cite{klavzar-2023} for the graphs $K(n,2)$ and $L(K_n)$ are completely mobile, so that $\cmob (K(n,2) = \mob (K(n,2))$ for $n \geq 5$ and $\cmob (L(K_n)) = \mob (L(K_n))$ for $n \geq 4$. For block graphs that contain a triangle one robot must be dropped to ensure complete mobility.
	
	\begin{proposition}\label{prop:blockgraph cmobmv}
		If $G$ is a block graph with clique number $\omega (G) \geq 3$, then \[ \cmob (G) = \cmobmv (G) = \omega (G) -1.\]
	\end{proposition}
	\begin{proof}
		Let $G$ be as stated. It follows from Theorem~\ref{thm:cliquevsmobmvnum} that $\cmobmv (G) \geq \omega (G) - 1$. A block graph is geodetic, so a subset is a mutual visibility set if and only if it is in general position, so that $\cmobmv (G) = \cmob (G)$. 
		
		For the upper bound, we show that we can gather all the robots into a clique. From the starting configuration, move the robots into a configuration that minimises the sum of the distances between them, and suppose for a contradiction that there are two robots $R_1$ and $R_2$ that are not adjacent in this configuration. Let $u_0,u_1,\dots ,u_r$ be the shortest path from $R_1$ to $R_2$; all of $u_1,\dots ,u_{r-1}$ are cut-vertices. One of the components of $G-u_{r-1}$ must contain all but one of the robots (say this component is the one containing $R_2$), so moving $R_1$ along $P$ to $u_{r-1}$ would keep the robots mutually visible at all times, and reduce the sum of the distances. Hence, if $\cmobmv (G) = \omega (G)$, we can suppose that the robots occupy a maximum clique $W$ on vertices $w_1,\dots ,w_{\omega }$. Then it is easily seen that, for any $1 \leq i \leq \omega (G)$, the only move available to the robot at $u_i$ is to move into one of the components of $G-u_i$ other than the component containing $W-u_i$. Whilst the robot is visiting these other components, no other robot can move to $u_i$; it follows that no other robot can ever visit $u_i$, and hence the configuration is not completely mobile. Thus $\cmobmv (G) = \omega (G) - 1$.
	\end{proof}
	
	It was an open question whether the difference $\mob (G) - \cmob(G)$ can be greater than one. However, it turns out that the difference between the two numbers can be arbitrarily large. We now characterise all possible combinations using a variation on half graphs.
	
	\begin{theorem}\label{thm:realisation comp mob gp}
		For any $r,s$ with $2 \leq s \leq r$ there is a graph with $\mob ^*(G) = s$ and $\mob (G) = r$.
	\end{theorem}
	\begin{proof}
		If $r = s$, the complete $(r+1)$-partite graph in which each part has two vertices has the required mobile and completely mobile general position numbers. We now show the result for $s = 2$ and $r > 2$ using the half graph defined as follows. Set $V(G(r)) = \{ a_i,b_i:1 \leq i \leq r\} $ and $E(G(r)) = \{ a_ib_j : 1 \leq i \leq j \leq r\} $. We claim that $G(r)$ has $\mob (G(r)) = r$ and $\mob ^*(G(r)) = 2$. 
		
		\begin{figure}[h!]
			\centering
			\begin{tikzpicture}[x=0.2mm,y=-0.2mm,inner sep=0.5mm,scale=1,thick,vertex/.style={circle,draw,minimum size=15}]
				\foreach \i in {1, 2, ..., 5}{
					\node at (\i*70-70,0) [vertex] (a\i) {$a_\i$};
					\node at (\i*70-70,70) [vertex] (b\i) {$b_\i$};
					\foreach \j in {1, ..., \i}{
						\path (a\j) edge (b\i);
					}
				}	
			\end{tikzpicture}
			\caption{The graph $G(5)$.}
			\label{fig:G(5)}
		\end{figure}
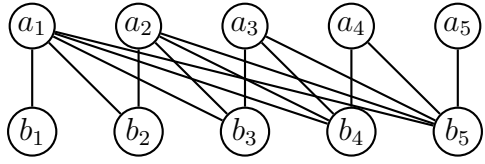  
		
		First we note that for $1 \leq i \leq j \leq r$ the set $\{ a_i,b_j\} $ is a maximal general position set of $G(r)$. Hence if there are at least three robots on $G(r)$, then the robots are always at vertices with distinct subscripts. This implies that $\mob (G(r)) \leq \gp (G(r)) = r$. To prove equality, station a robot $R_i$ on $b_i$ for $1 \leq i \leq r$ and perform the moves \[ b_r \move a_r, b_{r-1} \move a_{r-1}, \cdots ,b_1 \move a_1 \] in this order. It is easily seen that the robots are always in general position and every vertex has been visited after $r$ moves. 
		
		Now suppose for a contradiction that $G(r)$ has a completely mobile general position configuration with $t \geq 3$ robots $R_1,\dots ,R_t$. Suppose that initially $R_1$ has the smallest subscript, and another robot, say $R_2$, is the next robot to visit $b_1$. Then at some point $R_1$ has index $i$ and $R_2$ makes a move $b_j \move a_k$, where $1 \leq k < i < j \leq r$. $R_1$ cannot have been at $a_i$ before this move, since $\{ a_i,b_j\} $ is a maximal general position set of $G(r)$, so $R_1$ must have been at $b_i$. However, this implies that after the move of $R_2$, $R_1$ and $R_2$ occupy a set $\{ a_k,b_i\} $, where $k < i$, a contradiction. Hence $\mob ^*(G(r)) = 2$.
		
		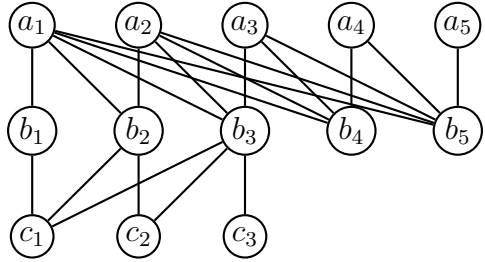
\begin{figure}[h!]
			\centering
			\begin{tikzpicture}[x=0.2mm,y=-0.2mm,inner sep=0.5mm,scale=1,thick,vertex/.style={circle,draw,minimum size=15}]
				\foreach \i in {1, 2, ..., 5}{
					\node at (\i*70-70,0) [vertex] (a\i) {$a_\i$};
					\node at (\i*70-70,70) [vertex] (b\i) {$b_\i$};
					\ifnum\i<4
					\node at (\i*70-70,140) [vertex] (c\i) {$c_\i$};
					\fi
					\foreach \j in {1, ..., \i}{
						\path (a\j) edge (b\i);
						\ifnum\i<4
						\path (c\j) edge (b\i);
						\fi
					}
				}	
			\end{tikzpicture}
			\caption{The graph $G(5,3)$.}
			\label{fig:G(5,3)}
		\end{figure}  
		
		Now suppose that $r > s > 2$. We define the graph $G(r,s)$ by adding a set of $s$ vertices $C = \{ c_i : 1\leq i \leq s \} $ to $G(r)$ and adding the edge $c_ib_j$ whenever $1 \leq i \leq j \leq s$. An example is shown in Figure~\ref{fig:G(5,3)}. We claim that $G(r,s)$ has $\mob (G(r,s)) = r$ and $\mob ^*(G(r,s)) = s$. Note that $\{ a_i,b_j\} $ is again a maximal general position set for $1 \leq i \leq j \leq r$, as is $\{ c_i,b_j\} $ when $1 \leq i \leq j \leq s$, so no two robots ever occupy such a pair of vertices.
		
		First, we show that $\mob(G(r,s)) = r$. To see that $\mob(G(r,s)) \geq r$, station a robot $R_i$ on $b_i$ for $1 \leq i \leq r$ and perform the moves $ b_r \move a_r, b_{r-1} \move a_{r-1}, \cdots ,b_1 \move a_1$ in this order, as in the previous part. From the starting configuration we can also perform the sequence of legal moves $b_s \move c_s,  b_{s-1} \move c_{s-1}, \cdots ,b_1 \move c_1$ in this order. The robots are always in general position and every vertex has been visited. Suppose for a contradiction that we have a mobile general position set with $\geq r+1$ robots on $G(r,s)$. As no pair of robots are stationed at vertices $a_i,b_i$ simultaneously for any $1 \leq i \leq r$, there must always be a robot in $C$. But as no pair of robots can be stationed on a set $\{ c_i,b_s\} $, $1 \leq i \leq s$, it follows that no robot can visit vertex $b_s$, a contradiction. Thus $\mob (G(r,s)) = r$.
		
		Next, we show that $\mob^*(G(r,s)) \geq s$. First we use Lemma~\ref{lem:mobility} to show that $(b_1,b_2,\cdots ,b_s)$ is a mobile gp-configuration. Starting from $(b_1,\cdots ,b_s)$, the robot $R_s$ can perform the sequence of moves $b_s \move a_s \move b_{s+1} \move a_{s+1} \move \cdots \move b_r \move a_r$. Also starting from $(b_1,\cdots ,b_s)$ the robots can perform the same sequences of moves as in the preceding part to visit the vertices of $\{ a_1,\dots, a_s,c_1,\dots ,c_s\} $.   
		
		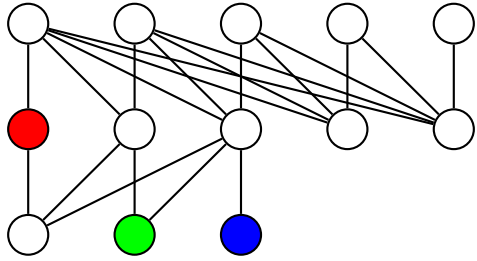
\begin{figure}[h!]
			\centering
			\begin{tikzpicture}[x=0.2mm,y=-0.2mm,inner sep=0.5mm,scale=1,thick,vertex/.style={circle,draw,minimum size=15}]
				
				\node at (0,70) [vertex, fill=red] (b1) {};
				\node at (70,70) [vertex] (b2) {};
				\node at (140,70) [vertex] (b3) {};
				\node at (210,70) [vertex] (b4) {};
				\node at (280,70) [vertex] (b5) {};
				
				\node at (0,140) [vertex] (c1) {};
				\node at (70,140) [vertex,fill=green] (c2) {};
				\node at (140,140) [vertex,fill=blue] (c3) {};
				
				\foreach \i in {1, 2, ..., 5}{
					\node at (\i*70-70,0) [vertex] (a\i) {};
					\foreach \j in {1, ..., \i}{
						\path (a\j) edge (b\i);
						\ifnum\i<4
						\path (c\j) edge (b\i);
						\fi
					}
				}	
			\end{tikzpicture}
			\caption{The graph $G(5,3)$ with the robots $R_1$ $R_2, R_3$ (in red, green and blue, respectively) after the first round of moves in the permutation.}
			\label{fig:G(5,3)perm1}
		\end{figure}   
		
		We now demonstrate how to permute the robots in the configuration $(b_1,\cdots ,b_s)$ in an $s$-cycle. Then we will be done by Lemma~\ref{lem:swapping}. First, move \[ b_s \move c_s, b_{s-1} \move c_{s-1}, \cdots ,b_2 \move c_2  \] in this order. See Figure~\ref{fig:G(5,3)perm1}. Following this, move $b_1 \move a_1 \move b_r$. Then perform the moves \[c_2 \move b_2 \move c_1 \move b_1, c_3 \move b_3 \move c_2 \move b_2, \cdots, c_s \move b_s \move c_{s-1} \move b_{s-1}\] in this order, see Figure~\ref{fig:G(5,3)perm2}. Finally, move $b_r \move a_s \move b_s$. We have now permuted the configuration $(b_1,b_2,\cdots ,b_s)$ to the configuration $(b_s,b_1,b_2,\cdots ,b_{s-1})$ by a sequence of legal moves, showing that $\mob^*(G(r,s)) \geq s$.
		
		\begin{figure}[h!]
			\centering
			\begin{tikzpicture}[x=0.2mm,y=-0.2mm,inner sep=0.5mm,scale=1,thick,vertex/.style={circle,draw,minimum size=15}]
				
				\node at (0,70) [vertex] (b1) {};
				\node at (70,70) [vertex] (b2) {};
				\node at (140,70) [vertex] (b3) {};
				\node at (210,70) [vertex] (b4) {};
				\node at (280,70) [vertex,fill=red] (b5) {};
				
				\node at (0,140) [vertex] (c1) {};
				\node at (70,140) [vertex,fill=green] (c2) {};
				\node at (140,140) [vertex,fill=blue] (c3) {};
				
				\foreach \i in {1, 2, ..., 5}{
					\node at (\i*70-70,0) [vertex] (a\i) {};
					\foreach \j in {1, ..., \i}{
						\path (a\j) edge (b\i);
						\ifnum\i<4
						\path (c\j) edge (b\i);
						\fi
					}
				}
				
				\draw [->,gray] (c2) to[bend right=30] node {} (b2);		
				\draw [->,gray] (b2) to[bend right=30] node {} (c1);	
				\draw [->,gray] (c1) to[bend left=30] node {} (b1);		
				
			\end{tikzpicture}
			\caption{The graph $G(5,3)$ with the second set of moves for $R_2$ ($c_2 \move b_2 \move c_1 \move b_1$) indicated with arrows.}
			\label{fig:G(5,3)perm2}
		\end{figure}

		Finally we prove the upper bound $\mob ^*(G(r,s)) \leq s$. Suppose that we have a completely mobile general position configuration of at least three robots. When a robot visits $b_1$, all other robots are at vertices with strictly larger indices. It follows that there is some point at which a robot $R_1$ has index less than or equal to the indices of all other robots, and at the next move a robot $R_2$ moves to a vertex with index strictly smaller than the other robots. To avoid $R_1$ and $R_2$ occupying a maximal general position set immediately before or after this move, $R_1$ must be at some $c_i$, $1 \leq i \leq s$, and $R_2$ at some $b_j$ with $s+1 \leq j \leq r$, and the move has the form $b_j \move a_k$, where $1 \leq k < i \leq s < j \leq r$. Before the move there are no robots in $\{ a_t: 1 \leq t \leq j\} \cup \{ b_t: i \leq t \leq s\} $. After the move there are no robots left in $\{ a_t,b_t: s+1\leq t \leq r\} $, since $c_ib_ia_kb_ta_t$ is a shortest path for $s+1 \leq t \leq r$. Hence, after the move there can be robots only at $a_k$ and in $\{ c_t : i \leq t \leq s\} $. We conclude that there are at most $s$ robots.
	\end{proof}
	
	\subsection{Completely mobile mutual visibility sets}\label{subsec:completely mobile mv}
	
	We begin our study of completely mobile mutual visibility sets by characterising the graphs with completely mobile mutual visibility number two.
	
	\begin{lemma}\label{mobmv>3}
		If $G$ is a connected graph with circumference $\cir (G) \geq 4$, then $\cmobmv{G} \geq 3$.
	\end{lemma}
	\begin{proof}
		Let $G$ be any graph with $\cir (G) \geq 4$. Let $C$ be the shortest cycle of $G$ that is not a triangle. It is easy to verify that $C$ is either isometric or is a 4-cycle (possibly with chords). In either case three robots can traverse $C$ and permute their positions. By reasoning similar to that of~\cite[Theorem 3.4]{dettlaff2024} the three robots are a mobile mutual visibility set, and by Lemma~\ref{lem:swapping} they are completely mobile.  
	\end{proof}
	
	\begin{corollary}
		If $G$ is a connected graph with $\mobmv (G) > 3$, then $\mobmv ^*(G) \geq 3$.    
	\end{corollary}
	\begin{proof}
		If $G$ is a tree, then $\mobmv (G) = 2$ by~\cite[Theorem 3.4]{dettlaff2024}. If $G$ has $\cir (G) = 3$, then $G$ is a block graph with $\omega (G) = 3$, and hence we would have $\mobmv{G} = 3$ by~\cite[Corollary 3.3]{dettlaff2024}. Otherwise the result follows by Lemma~\ref{mobmv>3}.
	\end{proof}

	\begin{theorem}\label{thm:cmobmv = 2}
		Let $G$ be a connected graph. Then $\cmobmv{G} = 2$ if and only if $G$ is a block graph with $\omega(G) \leq 3$, but not a path.
	\end{theorem}
	\begin{proof}
		Suppose that $\cmobmv{G} = 2$. By Lemma~\ref{mobmv>3}, the circumference of $G$ is $\cir (G) \leq 3$, and so $G$ is a block graph with $\omega (G) \leq 3$. But as $\cmobmv(G) = 1$ if and only if $G$ is a path, the result follows.
		
		Conversely, if $G$ is a tree, but not a path, then $\cmobmv{G} \leq \mobmv{G} = 2$ by~\cite[Theorem 3.4]{dettlaff2024} and $\cmobmv (G) \geq 2$ by Proposition~\ref{prop:comp mob small values}, so that $\cmobmv (G) = 2$. The result for block graphs with clique number three follows from Proposition~\ref{prop:blockgraph cmobmv}.
	\end{proof}
	
	We now prove a realisation result for the mobile and completely mobile mutual visibility numbers analogous to Theorem~\ref{thm:realisation comp mob gp}. 
	
	\begin{theorem}\label{cmobmvlower}
		For $r,s \geq 2$, there exists a graph $G$ with $\cmobmv(G) = s$ and $\mobmv(G) = r$ if and only if i) $r \geq s \geq 3$ or ii) $s = 2$ and $r \in \{ 2,3\} $.
	\end{theorem}
	\begin{proof}
		For any graph $G$, $\cmobmv (G) \leq \mobmv (G)$. Also if $s = 2$, then it follows from Theorem~\ref{thm:cmobmv = 2} and~\cite[Corollary 3.3]{dettlaff2024} that $r = 2$ or $3$, and these values are achieved by $K_{1,3}$ and $K_3$. Hence we must only show the existence of a graph with $\cmobmv(G) = s$ and $\mobmv(G) = r$ for $r \geq s \geq 3$. We have $\cmobmv (C_4) = \mobmv (C_4) = 3$, $\cmobmv (K_{3,3}) = \mobmv (K_{3,3}) = 4$ and also $\cmobmv (K_{r-1,4}) = \mobmv (K_{r-1,4}) = r$ for $r \geq 5$. If $r = s+1$, $s \geq 3$, then the complete graph on $K_r$ has the required values. Hence we can assume that $r \geq s+2$.
		
		Let $s \geq 3$. For $a \geq 2, b \geq 3$, construct a graph $G(a,b)$ by starting with the strong product $P_{a} \strongprod P_2$, adding a clique $W$ of order $b-1$ and adding all edges between $\{ (a,1),(a,2)\} $ and $W$. We will show that $\mobmv (G(a,b)) = a+b$ and $\cmobmv (G(a,b)) = b$. Then the result will follow on setting $a = r-s,b = s$. 
		
		We start with the mobile mutual visibility number. The set $([a] \times \{ 2\} ) \cup W \cup \{ (1,1)\} $ is a mutual visibility set, and for $2 \leq i \leq a$ the robot at $(i,2)$ can make the legal move $(i,2) \move (i,1)$, so $\mobmv (G(a,b)) \geq a+b$. If there are $> a+b$ robots on $G(a,b)$, then there must always be a robot in $W$ (for otherwise there would be at most $\mu (P_a \strongprod P_2) = a+2$ robots) and there would be at least two elements $i,j \in [a]$, $i < j$, with robots on each vertex of $\{ i,j\} \times [2]$. However then all shortest paths from $\{ i\} \times [2]$ to any robot in $W$ would pass through a robot in $\{ j\} \times [2]$).  Hence $\mobmv (G(a,b)) \leq \mu (G(a,b)) = a+b$ and equality holds. 
		
		Now we deal with the completely mobile mutual visibility number. By Theorem~\ref{thm:cliquevsmobmvnum} we have $\cmobmv (G(a,b)) \geq b$. Suppose for a contradiction that there is a completely mobile mutual visibility set of $G(a,b)$ with $\geq b+1$ robots. Note that there are always at least two robots outside of $W$. We take as our initial configuration the moment when some robot $R_1$ visits $(1,1)$. Out of all robots in $P_a \strongprod P_2$ apart from $R_1$, let $R_2$ be the robot with smallest first coordinate. For a robot other than $R_1$ or $R_2$ to visit $(1,1)$, there is a first move which results in a third robot $R_3$ having the same first coordinate as either $R_1$ or $R_2$ (without loss of generality, say it is $R_2$). However, at this point $R_2$ and $R_3$ will occupy a cut-set in $G(a,b)$, robot $R_1$ has strictly smaller first coordinate than $R_2$ and $R_3$, and there is a robot $R_4$ in $W$ or with strictly larger first coordinate than $R_3$. Hence $R_1$ and $R_4$ will not be mutually visible. Thus $\cmobmv (G(a,b)) = b$, completing the proof.
	\end{proof}
	The construction in Theorem~\ref{cmobmvlower} can also be used in a straightforward way to show that for any $r,s \geq 4$ there is a graph $G$ with $\mu (G) = r$ and $\gp (G) = s$ if and only if $r \geq s \geq 4$ (it was also pointed out by Di Stefano that here one can use a Cartesian product instead of a strong product before appending the clique~\cite{DiStefanocomment}). Computer search by Erskine~\cite{Erskine} has identified graphs with general position number three and mutual visibility number $r$ for all $3 \leq r \leq 7$, but the existence of a graph with $\gp (G) = 3$ and $\mu (G) \geq 8$ remains unknown. 
	
	\section{Line graphs of complete graphs}\label{sec:line graph complete graph}
	
	Mutual visibility sets in $L(K_n)$ were first discussed in~\cite{cicerone-2024}. It is convenient to picture these sets as subsets of the edge set of $K_n$. Following the notation of~\cite{cicerone-2024}, for any set $F \subseteq E(K_n)$ of edges of $K_n$, we will denote the corresponding vertices of $L(K_n)$ by $S_F$ and the subgraph of $K_n$ induced by $F$ by $(K_n)_F$. It was shown in~\cite{cicerone-2024} that $S_F$ is a mutual visibility set if and only if $(K_n)_F$ is $K_4$-free. It follows immediately from Tur\'{a}n's Theorem that largest mutual visibility sets of $L(K_n)$ correspond to the edge sets of the $3$-partite Tur\'{a}n graph $T(n,3)$. The authors of~\cite{dettlaff2024} showed that $\mobmv (L(K_n)) < \mu (K_n)$ for $n \geq 6$ and asked how large the difference $\mu (K_n)-\mobmv (K_n)$ can be. We answer this question, and in fact find the exact value of $\mobmv (L(K_n))$ for all $n$, using a stability result for Tur\'{a}n's Theorem from~\cite{amin-2013}.
	
	\begin{theorem}
		For $n \geq 6$, $\mobmv{L(K_n)} = \cmobmv{L(K_n)} = \lfloor{\frac{n^2}{3}}\rfloor - \lfloor{\frac{n}{3}}\rfloor + 1$.
	\end{theorem}
	\begin{proof}
		Let $n \geq 6$ and $F\subseteq E(K_n)$ be such that $S_F$ is a mobile mutual visibility set of $L(K_n)$. First, we show that $\mobmv(L(K_n)) \leq  \left \lfloor{\frac{n^2}{3}} \right \rfloor - \left \lfloor{\frac{n}{3}} \right \rfloor + 1$. From~\cite[Theorem 14]{amin-2013}, the largest size of a $K_4$-free graph that is not $3$-partite is $\left \lfloor{\frac{n^2}{3}}\right \rfloor - \left \lfloor{\frac{n}{3}}\right \rfloor + 1$. Therefore, we can assume that at every stage $(K_n)_F$ is $3$-partite.

		Let the three parts of the initial configuration of the robots be $A,B,C$ and suppose that the smallest of these parts has order $1 \leq \ell \leq \frac{n}{3}$. At some point robots must visit edges of $K_n$ that have both end-points within one of these sets, so we focus on the stage that a robot moves from an edge between two different parts $A,B,C$ to an edge of $K_n$ connecting two vertices of one of the parts; without loss of generality, the robot moves to the edge $a_1a_2$, where $a_1,a_2 \in A$. We denote the configuration before this move by $F$ and afterwards by $F'$. In $F'$, let $B'$ be the set of common neighbours of $a_1$ and $a_2$ in $B$, and $C'$ be the set of common neighbours of $a_1$ and $a_2$ in $C$. As $(K_n)_{F'}$ is $K_4$-free, there are no edges in $(K_n)_{F'}$ between $B'$ and $C'$. It follows that there are at least  \[ |B'||C'| + |B| - |B'| + |C| - |C'| - 1 \geq \min \{ |B|,|C|\}  - 1 \geq \ell - 1\] edges missing in $F$ between $A$ and $B \cup C$. The largest possible size of $(K_n)_F$ is therefore \begin{equation}
			\ell(n-\ell) + \left \lfloor{\frac{(n-\ell)^2}{4}} \right \rfloor - (\ell - 1).  \label{edgeineq}
		\end{equation}
		
		Differentiating with respect to $\ell $ (or, technically, two quadratics in $\ell $ that differ by $\frac{1}{4}$), we see that its maximum value is attained at one of nearest integers to $\frac{n-2}{3}$. It follows that 
		
		\begin{equation}
			\ell(n-\ell) + \left \lfloor{\frac{(n-\ell)^2}{4}} \right \rfloor - (\ell - 1) \leq \left \lfloor{\frac{n^2}{3}} \right \rfloor - \left \lfloor{\frac{n}{3}} \right \rfloor + 1,  
		\end{equation}
		thus completing the proof of the upper bound.

		It remains to show that a mobile mutual visibility set of order $\left \lfloor{\frac{n^2}{3}} \right \rfloor - \left \lfloor{\frac{n}{3}} \right \rfloor + 1$ exists. Write $n = 3q+r$, $0 \leq r \leq 2$. Begin with the subgraph $T(n,3)$ of $K_n$ and label the three parts of $T(n,3)$ as $A, B, C$, where $|A| = q$. Write $A = \{ a_1,\dots ,a_q\} $, $B = \{ b_1,\dots ,b_{|B|}\} $, $C = \{ c_1,\dots,c_{|C|} \} $. We let $F$ be the set of edges of $T(n,3)$ with the edges $b_1a_i$ deleted for $2 \leq i \leq q$. Then $|F| = \lfloor{\frac{n^2}{3}}\rfloor - \lfloor{\frac{n}{3}}\rfloor + 1$. As long as $(K_n)_F$ is $3$-partite, it will be $K_4$-free and $S_F$ will be mutually-visible. A missing edge $b_1a_i$, $2 \leq i \leq q$, of $T(n,3)$ can be visited from $F$ using the move $b_1a_1 \move b_1a_i$. It can also be easily seen that for any $x \in B \cup C$ we can move the edges of $F$ into a configuration $F_{x,j}$ given by deleting from $T(n,3)$ the edges $xa_i$ for $1 \leq i \leq q$ and $i \neq j$. 
		
		Finally we show how to visit edges of $K_n$ that join two vertices in the same part of $T(n,3)$. To visit an edge $b_ib_j$, start with the configuration $F_{b_i, 1}$ and make the move $a_1b_i \move b_ib_j$. As $b_i$ has no neighbours in $A$ after this move, the move does not create any copies of $K_4$. The process to visit an edge with both endpoints in $C$ is analogous. To visit $a_ja_i$ ($i \neq j$), begin with $F_{b_1,j}$ and perform the moves $a_jb_k \move b_1b_k$ for $2 \leq k \leq |B| - 1$, and finally make the move $a_jb_{|B|} \move a_ja_i$. At all times any adjacent vertices in $A$ or $B$ have no common neighbours in $B$ or $A$, respectively, so the subgraph remains $K_4$-free. Thus $F$ is a mobile mutual visibility set.
		
		Moreover, this configuration is completely mobile. We show how to permute the edges in $F$, and the result then follows from Lemma~\ref{lem:swapping}. Move $b_1c_{|C|} \move b_1a_2$ and then $b_1c_{|C| - i} \move b_1c_{|C| - i + 1}$ for $1 \leq i \leq |C| - 1$. Next, move $b_2c_1 \move b_1c_1$ and then $b_2c_i \move b_2c_{i-1}$ for $2 \leq i \leq |C|$. Continuing in this way through all vertices in $B$, we will finish with $b_{|B|}c_j$ free, where $j \in \{1, |C|\}$. Perform $a_1c_j \move b_{|B|}c_j$ and then follow a similar sequence to move all edges between $A$ and $C$ and leave $a_qc_k$ free where $k \in \{ 1, |C| \}$. Then move $a_qb_2 \move a_qc_k$ if $q$ is even and $a_qb_{|B|} \move a_qc_k$ otherwise. Following the same pattern to shift all of the edges in $F$ between $A$ and $B$ leaves $a_1b_1$ free. To finish, move $a_2b_1 \move a_1b_1$. At all times, only edges in $T(n,3)$ were occupied, so all moves were legal.
	\end{proof}

	\section{Conclusion}\label{sec:conclusion}
	
	We conclude the paper with some new open problems. Firstly, we observe that one aspect of our model of a swarm of robots is especially unrealistic, namely that only one robot moves at a time. We suggest that it would be of interest to relax this condition and allow the robots to move at the same time, or stay put.
	
	\begin{problem}
		How many robots can traverse the graph in general position/mutual visibility if more than one robot can move at a time?
	\end{problem}
	
	Returning to our picture of robots delivering groceries, it may also be observed that the main aim is to deliver to each house as quickly as possible. It could be that maximising the number of robots in the network could actually increase the time taken to visit each vertex. 
	
	\begin{problem}
		How long does it take for the robots in a mobile general position/mutual visibility configuration to visit all the vertices the graph? If the configuration is completely mobile, how long does it take for every robot to visit every vertex? Is the fastest solution always attained by a configuration with the largest number of robots?
	\end{problem}
	
	Conjecture~\ref{conj:maxsize mobile num 2} also suggests investigating the following extremal problem in greater detail.
	
	\begin{problem}
		What is the largest size of a graph with order $n$ and mobile general position number $k$?
	\end{problem}
	
	Theorem~\ref{the:clique num 5} showed that for any graph $G$ with $\omega(G) \geq 5$, $\mob(G) \geq 3$. It is not clear whether this lower bound can be met for larger clique numbers.
	
	\begin{problem}
		Are there graphs with arbitrarily large clique number and mobile general position number three?
	\end{problem}

	\section*{Acknowledgements}
	Aoise Evans conducted this research during a research internship that was supported through EPSRC Doctoral Training Partnership DTP 2224 Open University. Ethan Shallcross and Aditi Krishnakumar were funded by Open University research bursaries. Sumaiyah Boshar worked on this project as part of a Nuffield Research Placement at the Open University. The authors would like to thank Grahame Erskine for his help with computations. James Tuite also received funding for this project from EPSRC grant [EP/W522338/1].
	


\begin{thebibliography}{99}
		\bibliographystyle{plain}
		
		\bibitem{amin-2013}
		K.~Amin, J.~Faudree, R.J.~Gould, E.~Sidorowicz,
		On the non-$(p-1)$-partite $K_p$-free graphs. Discuss. Math. Graph Theory 33 (1) (2013) 9--23.
		
		\bibitem{cicerone-2023}
		S.~Cicerone, A.~Di Fonso, G.~Di Stefano, A.~Navarra, F.~Piselli,
		Mutual visibility in hypercube-like graphs. 
		In International Colloquium on Structural Information and Communication Complexity (pp. 192-207). Cham: Springer Nature Switzerland. arXiv:2308.14443 [math.CO] (28 Aug 2023).
		
		\bibitem{cicerone-2024}
		S.~Cicerone, G.~Di Stefano, S.~Klav\v{z}ar, I.G.~Yero, Mutual-visibility problems on graphs of diameter two, European J. Combin. 120 (2024) article 103995.
		
		\bibitem{survey}
		U.~Chandran S.V., S.~Klav\v{z}ar, J.~Tuite, 
		The general position problem: a survey.
		\url{arXiv:2501.19385} (2025).
		
		\bibitem{ullas-2016}
		U.~Chandran S.V., G.J.~Parthasarathy,
		The geodesic irredundant sets in graphs.
		Int.\ J.\ Math.\ Combin.\  4 (2016) 135--143.
		
		\bibitem{Graphsdigraphs} G.~Chartrand, H.~Jordon, V.~Vatter, P.~Zhang, Graphs \& Digraphs. Chapman and Hall/crc, 2024.
		
		\bibitem{dettlaff2024} M.~Dettlaff, M.~Lema\'{n}ska, J.A.~Rodr\'{i}guez-Vel\'{a}zquez, I.G.~Yero, Mobile mutual-visibility sets in graphs. Ars Math. Contemp. (2025) 1--21.
		
		\bibitem{DiStefano} G.~Di Stefano, Mutual visibility in graphs. Appl. Math. Comput. 419 (2022) 126850.
		
		\bibitem{DiStefanocomment} G.~Di Stefano, personal communication (2024).
		
		\bibitem{Erskine} G.~Erskine, personal communication (2025).
		
		\bibitem{Cartesian} S.~Klav\v{z}ar, A.~Krishnakumar, D.~Kuziak, E.~Shallcross, J.~Tuite, I.G.~Yero, Moving through Cartesian products, coronas and joins in general position. arXiv preprint arXiv:2505.00535 (2025).
		
		\bibitem{klavzar-2023} 
		S.~Klav\v{z}ar, A.~Krishnakumar, J.~Tuite, I.G.~Yero, 
		Traversing a graph in general position.
		Bull.\ Aust.\ Math.\ Soc.\ 108 (2023) 353--365.
		
		\bibitem{Klavzar-2019}
		S.~Klav\v{z}ar, I.G.~Yero,
		The general position problem and strong resolving graphs,
		Open Math.\ 17 (2019) 1126--1135.
		
		\bibitem{manuel-2017arxiv} 
		P.~Manuel, S.~Klav\v{z}ar, 
		Graph theory general position problem.
		\url{arXiv:1708.09130} (2017).
	\end{thebibliography}
\end{document}